\documentclass[10pt]{amsart}
\usepackage{latexsym}
\usepackage{amsmath}
\usepackage{amssymb}
\usepackage{mathrsfs}
\usepackage{graphicx}
\usepackage{color}
\usepackage{pgfpages}
\usepackage{ifthen}
\usepackage{leftidx,tensor}
\usepackage[T1]{fontenc}
\usepackage[latin1]{inputenc}
\usepackage{mathtools}
\usepackage{comment}
\usepackage{dsfont}
\usepackage{wasysym}
\usepackage[compress,sort]{cite}
\usepackage[foot]{amsaddr}

\usepackage[shortlabels]{enumitem}
\usepackage{aliascnt}
\usepackage{esint}
\usepackage[bookmarks=true,pdfstartview=FitH, pdfborder={0 0 0}, colorlinks=true,citecolor=red, linkcolor=blue]{hyperref}
\usepackage{bbm}
\usepackage{nicefrac}

\theoremstyle{plain}
\newtheorem{thm}{Theorem}[section]
\newaliascnt{cor}{thm}
\newaliascnt{prop}{thm}
\newaliascnt{lem}{thm}
\newtheorem{cor}[cor]{Corollary}
\newtheorem{prop}[prop]{Proposition}
\newtheorem{lem}[lem]{Lemma}
\aliascntresetthe{cor}
\aliascntresetthe{prop}
\aliascntresetthe{lem} 
\theoremstyle{definition}
\newaliascnt{defn}{thm}
\newaliascnt{asu}{thm}
\newaliascnt{con}{thm}
\aliascntresetthe{defn}
\aliascntresetthe{asu}
\aliascntresetthe{con}
\newcounter{stp}
\newcounter{stpi}
\newcounter{stpci}
\newcounter{stpiii}

\theoremstyle{thm}
\newaliascnt{rem}{thm}
\newaliascnt{exa}{thm}
\newaliascnt{masu}{thm}
\newaliascnt{nota}{thm}
\newaliascnt{sett}{thm}
\newtheorem{rem}[rem]{Remark}

\aliascntresetthe{rem}
\aliascntresetthe{exa}
\aliascntresetthe{masu}
\aliascntresetthe{nota}
\aliascntresetthe{sett}
\numberwithin{equation}{section}

\setlist[enumerate]{font = \normalfont}

	\newcommand{\N}{\mathbb{N}}

	\newcommand{\R}{\mathbb{R}}
	\newcommand{\C}{\mathbb{C}}
	
	\newcommand{\rC}{\mathrm{C}}
	\newcommand{\rL}{\mathrm{L}}
	\newcommand{\rW}{\mathrm{W}}
	\newcommand{\rH}{\mathrm{H}}
	\newcommand{\rB}{\mathrm{B}}

	\newcommand{\rQ}{\mathrm{Q}}

	\renewcommand{\div}{\mathrm{div} \,}

    \renewcommand{\Re}{\mathrm{Re}}
	
	\title{The $\mathrm{L}^1$-Stokes Semigroup}
	\author[Tim Binz \& Anatole Gaudin]{Tim Binz$^\dagger$ \& Anatole Gaudin$^{\ast}$}
	\address{$^\dagger$ TU Darmstadt\\
		Department of Mathematics \\
        Schlossgartenstr. 7 \\
        64289 Darmstadt \\ Germany.
		}
	\email{binz@mathematik.tu-darmstadt.de}
    
	\address{$^\ast$ Universit\"{a}t Duisburg-Essen\\
		Fakult\"{a}t f\"{u}r Mathematik \\
		Thea-Leymann-Stra{\ss}e 9 \\
		45127 Essen \\ Germany.
		}
	\email{anatole.gaudin@uni-due.de}

	
\begin{document}

\begin{abstract}
    We study the Stokes operator with no-slip boundary conditions on the spaces $\rL^1_{\sigma,n}(\Omega)$ and $\nicefrac{\rL^1(\Omega,\C^d)}{\nabla \rW^{1,1}(\Omega,\C)}$, where $\Omega\subset\R^d$ is an arbitrary bounded $\rC^{1,\alpha}$-domain. We show that the Stokes operator on $\rL^1_{\sigma,n}(\Omega)$ does not generate a $\rC_0$-semigroup, even though the resolvent problem is uniquely solvable. In stark contrast, its realization on $\nicefrac{\rL^1(\Omega,\C^d)}{\nabla \rW^{1,1}(\Omega)}$ generates a compact, analytic $\rC_0$-semigroup, which leaves $\rL^1_{\sigma,n}(\Omega)$ invariant.

The key point is that these two realizations, which are canonically identified for $1<p<\infty$ through the Helmholtz decomposition, cease to be equivalent at the endpoint $p=1$. This leads to genuinely different functional analytic properties. Our result provides the first positive generation theorem for the Stokes operator with no-slip boundary conditions in a pure $\rL^1$-setting on a bounded domain and settles a problem that had remained open for nearly fifty years; see, e.g., \cite{Koz:01,DHP:01}. In this sense, it completes the theory of the Stokes semigroup across the full scale of solenoidal Lebesgue spaces on (smooth) bounded domains.  As an intermediate step, some results on the space of Radon measures are obtained.

The proof combines the sun-dual construction with a precise analysis of the failure of the Helmholtz decomposition in $\rL^1$, the celebrated result of Abe and Giga \cite{AG:12} on the Stokes semigroup on $\rC_{\sigma,0}(\Omega)$ and the regularity theory refinements for the Stokes operator recently developed by Breit and the second author \cite{BG:25}.

\end{abstract} 
	 
	\maketitle		

\section{Introduction}
\label{sec:intro}
%

On bounded domains, the generation of an analytic $\rC_0$-semigroup by the Stokes operator on the solenoidal vector fields $\rL^p_{\sigma,n}(\Omega)$ is a classical problem. For $p \in (1,\infty)$, it was established independently by Giga \cite{Gig:81} and Solonnikov \cite{Sol:77}.

These results were later extended to unbounded domains by Farwig, Kozono, and Sohr \cite{FKS:05}. 
In this setting, the failure of the Helmholtz decomposition in $\rL^p(\Omega)$ requires replacing $\rL^p(\Omega)$ by $\rL^p(\Omega)\cap \rL^2(\Omega)$ for $p\in(2,\infty)$, and by $\rL^p(\Omega)+\rL^2(\Omega)$ for $p\in(1,2)$. 
 Following these breakthrough results, the investigation of the functional-analytic properties of the Stokes operator on $\rL^p$-spaces emerged as an active research field, see e.g. \cite{Uka:87,Gig:85,GS:91,NS:03,KKW:06,MM:08,KW:13,KW:17,AG:14,FKS:07,FKS:09,She:12,Tol:17,Tol:20,GT:22,GS:25a,GHHSS:10}.

\smallskip 
For the endpoint cases $p \in \{1,\infty\}$, the situation is considerably more delicate due to the absence of a Helmholtz decomposition. In particular, the case $p=\infty$ remained open for nearly forty years. In their celebrated work, Abe and Giga \cite{AG:12} proved that the Stokes operator generates an analytic $\rC_0$-semigroup on $\rC_{\sigma,0}(\Omega)$ and, as a consequence, an analytic but non-$\rC_0$-semigroup on $\rL^\infty_{\sigma,n}(\Omega)$ for bounded $\rC^3$-domains.

An alternative proof was later provided by Abe, Giga, and Hieber \cite{AGH:13}. Very recently, Geng and Shen \cite{GS:25} extended the result to bounded $\rC^1$-domains, and Lipschitz domains if $d=2$.
For subsequent work about the Stokes semigroup on $\rL^\infty_{\sigma,n}(\Omega)$ see e.g. \cite{Abe:14,Abe:16,Abe:20,Abe:21,BvB:16,BH:15}.

\smallskip 

For $p=1$, much less is known. In \cite{DHP:01}, Desch, Hieber, and Pr\"{u}ss showed that, in the case of the half-space~$\R^d_+$, the Stokes operator does not generate an analytic $\rC_0$-semigroup on $\rL^1_{\sigma,n}(\R^d_+)$. In \cite{Saa:07} this result was extended to Robin boundary conditions. 
In \cite{BvB:16} von Below and Bolkart proved that, in the case of an infinite later $\R^{d-1} \times (0,\delta)$, the Stokes operator does generate an analytic $\rC_0$-semigroup on $\rL^1_{\sigma,n}(\R \times (0,\delta))$, whereas for $d \geq 3$ it does not.

Because of its more favourable behaviour with respect to singular integral operators, the space $\rL^1$ has often been replaced by the Hardy space $\rH^1$. In this setting, the Hardy space $\rH^1$ has been used successfully as a surrogate for $\rL^1$-type decay semigroup estimates by Miyakawa \cite{Miy:96}, Kozono \cite{Koz:98, Koz:01}, and Giga, Matsui, and Shimizu \cite{GMS:99}. For another application of $\rL^1$ decay estimates, see also the work of Maremonti \cite{Mare:11}. 
For an exhaustive review of the Stokes operator we refer to the survey of Hieber and Saal \cite{HS:18} and the recent article of Breit and the second author \cite{BG:25}.

Nevertheless, nearly fifty years after the groundbreaking works of Giga \cite{Gig:81} and Solonnikov \cite{Sol:77}, it still  remained a long-standing open problem, whether the Stokes operator on $\rL^1_{\sigma,n}(\Omega)$, provided $\Omega$ is a bounded domain, generates an analytic $\rC_0$-semigroup, see for instance \cite{Koz:98} and \cite[Remark 5.2]{DHP:01}. The following result provides a negative answer to this question.

\begin{thm}[{Simplified version of \autoref{cor:Stokes Operator on L1} and \autoref{cor:L_1-stokes}}]\label{cor:main}
	Let $\Omega \subset \R^d$ be a bounded $\rC^{1,\alpha}$-domain, $\alpha>0$.
	The Stokes operator $A \colon D_1(A) \subset \rL^1_{\sigma,n}(\Omega) \to \rL^1_{\sigma,n}(\Omega)$ with no-slip boundary conditions defined by
	\begin{equation*}
		A u =- \Delta u + \nabla \pi,  \quad \text{ with domain } \quad D_1(A) = \{ u \in \rW^{1,1}_0(\Omega) \cap \rL^1_{\sigma,n}(\Omega) \ \colon \exists \pi\in\rL^1(\Omega),\ \mathrm{ s.t. }\ \Delta u - \nabla \pi \in  \rL^1_{\sigma,n}(\Omega) \}
	\end{equation*}
	is densely defined, closed and invertible with
    \begin{align*}
        \sigma(A)\subset(0,\infty)
    \end{align*}
    but is \textbf{not} the (negative) generator of a $\rC_0$-semigroup on  $\rL^1_{\sigma,n}(\Omega)$. 
\end{thm}

The main obstruction is the failure of the Helmholtz decomposition in $\rL^1(\Omega,\C^d)$. 
Indeed, in $\rL^p(\Omega,\C^d)$ with $p \in (1,\infty)$ the Helmholtz decomposition is equivalent to the solvability of the weak Neumann problem \cite{SS:92, Gal:11} and, provided $\Omega$ has a smooth enough boundary, the latter one implies bounded $\rH^\infty$-calculus \cite{GK:15} and maximal regularity \cite{GHHS:12} for the Stokes operator on $\rL^p_{\sigma,n}(\Omega)$, and subsequently generation of an analytic $\rC_0$-semigroup.
We therefore briefly recall how the Helmholtz decomposition in $\rL^p(\Omega,\C^d)$ for $p \in (1,\infty)$ is commonly used, often implicitly, to identify different realizations of the Stokes operator. 
%
Starting from the expression
\begin{equation*}
	A u = -\Delta u + \nabla \pi
\end{equation*}
for a pressure term $\pi \in \rW^{1,p}(\Omega)$, we have to restrict it to a subset of $\rL^p(\Omega,\C^d)$ in order to handle the pressure term. One deals with $u$ ``up to the gradient part'', i.e. one considers 
\begin{equation*}
	A (u + \nabla p) = -\Delta u + \nabla \widetilde{\pi} ,
\end{equation*}
where $\nabla\widetilde{\pi}=\nabla[(-\Delta)p + \pi]$ or more precisely
\begin{equation*}
	\tilde{A} [u] = [-\Delta u + \nabla \pi], 
\end{equation*}
where $[ u ] = u + \nabla \rW^{1,p}(\Omega)$ with domain
\begin{equation*}
	D_{{p}}(\tilde{A}) = \{ [u] \in \nicefrac{\rL^p(\Omega,\C^d)}{\nabla \rW^{1,p}(\Omega)} \colon \exists (\widetilde{u},\pi)\in[u]\times\rL^{p}(\Omega) \ \mathrm{ s.t. }\ \widetilde{u}\in\rW^{1,p}_{\sigma,0}(\Omega,\C^d) \ \mathrm{ and }\ \Delta \widetilde{u} - \nabla \pi \in\rL^p(\Omega,\C^d) \} .
\end{equation*}
Strictly speaking this Stokes operator $\tilde{A}$ is an operator on the quotient space $\nicefrac{\rL^p(\Omega,\C^d)}{\nabla \rW^{1,p}(\Omega)}$. The Helmholtz decomposition
\begin{equation}
	\rL^p(\Omega,\C^d) = \rL^p_{\sigma,n}(\Omega) \oplus \nabla \rW^{1,p}(\Omega)
	\label{eq:helmholtz decomposition}
\end{equation}
implies that the quotient space identification
\begin{equation}\label{eq:Helmholtz quotient isomorphism}
    \nicefrac{\rL^p(\Omega,\C^d)}{\nabla \rW^{1,p}(\Omega)} \cong \rL^p_{\sigma,n}(\Omega,\C^d)
\end{equation}
i.e., it can be canonically identified with the solenoidal vector-field. In this way one obtains an operator $A \colon D_p(A) \subset \rL^p_{\sigma,n}(\Omega,\C^d) \to \rL^p_{\sigma,n}(\Omega,\C^d)$ given by
\begin{equation}
	A u = -\mathbb{P} \Delta u \label{Stokes Wrong definition}
\end{equation}
where $\mathbb{P} \colon \rL^p(\Omega,\C^d) \to \rL^p_{\sigma,n}(\Omega,\C^d)$ denotes the Leray projection, i.e. the projection to the first summand in the Helmholtz decomposition \eqref{eq:helmholtz decomposition}. Here it matters that $\Delta$, subject to prescribed boundary conditions, does not commute with gradients on $\Omega$ due to the no-slip boundary conditions. Thus, if $\Omega$ has smooth boundary, $-\mathbb{P} \Delta u$ selects the solenoidal representative of the class $[-\Delta u]$.

\smallskip 

This is well-defined and independent of the choice of representatives for both $[u]$ and $\tilde{A}[u]=[-\Delta u+\nabla \pi]$. Indeed since $\rL^{p'}_{\sigma,n}(\Omega)=\overline{\rC^\infty_{\sigma,c}(\Omega)}^{\lVert \cdot\rVert_{\rL^{p'}(\Omega)}}$, and $\rL^{p'}_{\sigma,n}(\Omega)'\cong \nicefrac{\rL^p(\Omega,\C^d)}{\nabla \rW^{1,p}(\Omega)}$, $p\in(1,\infty)$, if $p_j,\pi_j\in\mathcal{D}'(\Omega)$, $j=0,1$, are such that $u+\nabla p_j \in\rW^{1,p}_{\sigma,0}(\Omega,\C^d)$, $\Delta ({u}+\nabla p_j) - \nabla \pi_j \in\rL^p(\Omega,\C^d)$, we can use $\rC^\infty_{\sigma,c}(\Omega)$ as the space of test functions and
\begin{align*}
     \langle -\Delta ({u}+\nabla p_0) + \nabla \pi_0, \varphi \rangle =  \langle -\Delta {u}, \varphi \rangle=  \langle -\Delta ({u}+\nabla p_1) + \nabla \pi_1, \varphi \rangle 
\end{align*}
so that it is independent of choices of representatives and  taking the supremum over $\varphi \in \rC^\infty_{\sigma,c}(\Omega)$ with $\|\varphi\|_{\rL^{p'}(\Omega)}\leqslant 1$:
\begin{align*}
   \lVert \tilde{A}[u]\rVert_{\nicefrac{\rL^p(\Omega,\C^d)}{\nabla \rW^{1,p}(\Omega)}}&=\lVert [-\Delta ({u}+\nabla p_0) + \nabla \pi_0]\rVert_{\nicefrac{\rL^p(\Omega,\C^d)}{\nabla \rW^{1,p}(\Omega)}}\\&=\lVert [-\Delta ({u}+\nabla p_1) + \nabla \pi_1]\rVert_{\nicefrac{\rL^p(\Omega,\C^d)}{\nabla \rW^{1,p}(\Omega)}}.
\end{align*}
In general, unless one can systematically select a representative $u\in\rW^{2,p}(\Omega,\C^d)$, nothing ensures  that $[-\Delta {u}]\in\nicefrac{\rL^p(\Omega,\C^d)}{\nabla \rW^{1,p}(\Omega)}$   while $[-\Delta ({u}+\nabla p) + \nabla \pi]=[-\Delta u]$ in $\rC^\infty_{\sigma,c}(\Omega)'$. This is due  to the fact that, \textit{a priori}, one only has $\nabla(-\Delta p+\pi)\in[\rW^{-2,p}+\rW^{-1,p}](\Omega,\C^d)\not\subset \rL^p(\Omega,\C^d)$. This is also for the exact same reason as why one cannot usually write the Stokes operator as \eqref{Stokes Wrong definition} when the domain $\Omega$ does not have a smooth enough boundary regularity.

\smallskip
\

However, this entire construction, arising from the identification \eqref{eq:Helmholtz quotient isomorphism}, breaks down for $p = 1$ due to the failure of the Helmholtz decomposition, see \cite{BB:02},
\begin{equation}
	\rL^1(\Omega,\C^d) \not =  \rL^1_{\sigma,n}(\Omega) \oplus \nabla \rW^{1,1}(\Omega) .
	\label{eq:no helmholtz decomposition L1}
\end{equation}
Therefore, we are left with two distinct Stokes operators $A \colon D_1(A) \subset \rL^1_{\sigma,n}(\Omega) \to \rL^1_{\sigma,n}(\Omega)$ and $\tilde{A} \colon D_{1}(\tilde{A}) \subset \nicefrac{\rL^1(\Omega,\C^d)}{\nabla \rW^{1,1}(\Omega)}
\to \nicefrac{\rL^1(\Omega,\C^d)}{\nabla \rW^{1,1}(\Omega)}$. The non-existence of a bounded Leray projection on $\rL^1(\Omega,\C^d)$ is already a red flag for the definition of the first operator and ultimately the reason why it cannot generate a $\rC_0$-semigroup. 
Surprisingly, the quotient  realization  of the Stokes operator  is much better behaved.  

\begin{thm}[Simplified version of {\autoref{thm:sun-stokes}}]\label{thm:main}
	Let $\Omega \subset \R^d$ be a bounded $\rC^{1,\alpha}$-domain, $\alpha > 0$. 
	The Stokes operator $\tilde{A} \colon D_{1}(\tilde{A}) \subset \nicefrac{\rL^1(\Omega,\C^d)}{\nabla \rW^{1,1}(\Omega)} \to \nicefrac{\rL^1(\Omega,\C^d)}{\nabla \rW^{1,1}(\Omega)}$  with no-slip boundary conditions defined~by
	\begin{align*}
		\tilde{A} [u] &= [-\Delta u+\nabla \pi], 
		\quad \text{ with domain } \\
        D_{1}(\tilde{A}) &= \{ [u],\, u\in \rW^{1,1}_{0}(\Omega) \cap \rL^1_{\sigma,n}(\Omega) \colon \exists \pi\in\rL^1(\Omega)\, \textrm{s.t.}\, [\Delta u-\nabla \pi] \in  \nicefrac{\rL^1(\Omega,\C^d)}{\nabla \rW^{1,1}(\Omega)} \}
	\end{align*}
	is the (negative) generator of a bounded, compact and analytic $\rC_0$-semigroup $(\tilde{T}(t))_{t\geqslant 0}$ of angle $\frac{\pi}{2}$ on $\nicefrac{\rL^1(\Omega,\C^d)}{\nabla \rW^{1,1}(\Omega)}$. Moreover, we have the continuous embedding $\iota \colon \rL^1_{\sigma,n}(\Omega) \hookrightarrow \nicefrac{\rL^1(\Omega,\C^d)}{\nabla \rW^{1,1}(\Omega)}$ with dense range and the semigroup $(\tilde{T}(t))_{t \geqslant 0}$ leaves $\iota (\rL^1_{\sigma,n}(\Omega))$ invariant.
\end{thm}

Note that \autoref{thm:main}, together with the closed graph theorem, implies that the restrictions of $(\tilde{T}(t)){t \geqslant 0}$ to $\iota(\rL^1{\sigma,n}(\Omega))$ form a semigroup of $\rL^1$-bounded operators on $\iota(\rL^1_{\sigma,n}(\Omega))$. By \autoref{cor:main}, this semigroup is neither strongly continuous nor analytic with respect to the $\rL^1$-topology. However, it is strongly continuous and analytic with respect to the finer topology induced by $\nicefrac{\rL^1(\Omega,\C^d)}{\nabla \rW^{1,1}(\Omega)}$.

\smallskip 

Most of the difficulties arise from the failure of the Helmholtz decomposition and the absence of a bounded Leray projection, due to the fact that $\nabla\rW^{1,1}(\Omega)$ is not complemented in $\rL^{1}(\Omega,\C^d)$, see Bourgain and Brezis \cite{BB:02}.
The classical approach to generation results on domains relies on localization together with a careful analysis of the half-space problem. However, this strategy is ruled out by the results of Desch, Hieber, and Pr\"{u}ss \cite{DHP:01} and those of Uka\"{i} \cite{Uka:87}.

Another major difficulty is that operators on quotient spaces are generally awkward to handle. In particular, the semigroup generated by 
$-\tilde{A}$
cannot be realized as the quotient semigroup associated with the heat semigroup on $\rL^1(\Omega,\C^d)$, since the latter does not preserve the subspace of gradients $\nabla \rW^{1,1}(\Omega)$. This failure ultimately stems from the incompatibility with the no-slip boundary conditions.

This is not merely a technical obstruction, but rather reflects an intrinsic feature of the Stokes semigroup on domains: unlike in the whole-space or perfect slip boundary conditions settings, it cannot simply be viewed as a restriction or quotient of the heat semigroup.

\smallskip

Rather than relying on localization arguments or quotient semigroup techniques, our approach is based on a refined duality argument.

The starting point is the Stokes operator $A$ on $X=\rC_{\sigma,0}(\Omega)$. By results of Abe, Giga \cite{AG:12} and Geng, Shen \cite{GS:25}, the operator $A$ generates a compact analytic $\rC_0$-semigroup on $X$. We then consider the associated sun-dual operator \(A^{\odot}\) and the sun-dual space $X^{\odot}$. Recall that $X^{\odot}\subset X'$ is the largest subspace on which the dual semigroup acts strongly continuously.

Since $A$ has compact resolvent, a classic result of Phillips, see \cite{Phi:55} or \autoref{lem:sun dual reflexiv}, implies that the sun-bidual satisfies
$
X^{\odot\odot}\cong X.
$
In particular, the sun-predual and sun-dual constructions coincide in this setting. This suggests that, despite the fact that the full dual space $X'$ is very large, the sun-dual space $X^{\odot}$ associated with the Stokes semigroup may still admit an $\rL^1$-type description.

Indeed, in \autoref{lem:X'} we prove that
$X'
\cong
\nicefrac{M(\Omega,\C^d)}{\nabla \mathrm{BV}(\Omega)},$
where $M(\Omega,\C^d)$ denotes the space of vector-valued Radon measures and $\mathrm{BV}(\Omega)$ the space of functions of bounded variation. Combining this representation with the identity
$X^{\odot}=\overline{D(A')}^{\,X'},$
see \autoref{lem:sun dual closure}, and regularity theory for the Stokes operator on $\rL^p_{\sigma,n}(\Omega)$, we identify the sun-dual space as
$X^{\odot}\cong
\nicefrac{\rL^1(\Omega,\C^d)}{\nabla \rW^{1,1}(\Omega)},$
see \autoref{prop:sun dual stokes}.

Having established this characterization, we show that the corresponding quotient realization of the Stokes operator coincides with the sun-dual operator $A^{\odot}$. The abstract theory of sun-dual operators, recalled in \autoref{prop:sun dual}, then yields \autoref{thm:main}. The invariance of $\iota (\rL^1_{\sigma,n}(\Omega))$ follows from the theory developed in \cite{BG:25}. Finally, we obtain the non-generation result stated in \autoref{cor:main} by using scaling techniques and \cite[Theorem 5.1]{DHP:01}.

\smallskip 

Despite sun-dual semigroups being a standard construction in semigroup theory, we are not aware of any previous use of this concept in the context of Stokes operators. This seems to be one of the first applications to a ``concrete non-trivial'' operator. Likewise, the explicit distinction between the Stokes operator defined on solenoidal vector fields and the Stokes operator viewed as an operator on the quotient space modulo gradient fields appears to be new.
Identifying the latter as a dual operator is common in $\rL^p$-spaces for $p \in (1,\infty)$, at least implicitly, when one identifies it with the operator on solenoidal vector fields, but to the best of our knowledge this perspective is novel in the endpoint cases.

The fact that the Stokes operator on solenoidal vector fields and the corresponding operator on the quotient space modulo gradient fields may differ, and can in fact exhibit completely different properties, was entirely unexpected to us and, at least from our perspective, rather surprising.

\smallskip 

While \autoref{cor:main} completes the operator-theoretic picture of the Stokes operator across the full scale of classical solenoidal Lebesgue spaces $\rL^p_{\sigma,n}(\Omega)$ for $p \in [1,\infty]$, \autoref{thm:main} establishes the sun-dual space of $\rC_{\sigma,0}(\Omega)$ as the canonical, well-behaved alternative at the $p=1$ endpoint. Although this abstract construction is highly robust, its specific $\rL^1$-characterization cannot be seamlessly generalized to unbounded domains.
In such geometries, even for $p \in (1,\infty)$, Farwig, Kozono, and Sohr \cite{FKS:05} already demonstrated that the Helmholtz decomposition fails and $\rL^p_{\sigma,n}(\Omega)$ must be replaced by appropriate intersections or sums with $\rL^2$.
Furthermore, the loss of compactness of the resolvent on $\rC_{\sigma,0}(\Omega)$ prevents the use of Phillips' sun-reflexivity theorem \cite{Phi:55}. This suggests that in the unbounded setting, the sun-dual construction may not reduce to a simple $\rL^1$-type quotient, but may instead yield a more exotic topological space. We plan to investigate this in the future.   

\smallskip 

Direct applications of our main \autoref{thm:main} are a new convergence result towards initial data in the space $\nicefrac{\rL^1(\Omega,\C^d)}{\nabla \rW^{1,1}(\Omega)}$ for solutions of the incompressible Navier--Stokes equations by using an iterative scheme, and new decay estimates for solutions. 
We plan to investigate these in the future works.  

\smallskip 

The present article is organized as follows. In \autoref{sec:prelim} we introduce useful notations and recall and comment on known results which we require later in the proof. In particular, in \autoref{ssec:sun dual} we recall the construction and theory of the sun-dual semigroup. The \autoref{ssec:Radon Meas. and BV} recalls elementary facts on Radon measures and the space of functions of Bounded Variations. Furthermore, in \autoref{ssec:stokes lp} we recall operator-theoretic results about the Stokes operator on $\rL^p_{\sigma,n}(\Omega)$ and $\rC_{\sigma,0}(\Omega)$.
Finally, \autoref{sec:main part} is the main section of the article. In \autoref{ssec:L1 no helmholtz} we discuss the failure of the Helmholtz decomposition in $\rL^1(\Omega,\C^d)$ and its consequences, 
in \autoref{ssec:sun dual space} we determine the sun-dual space of the Stokes operator on $\rC_{\sigma,0}(\Omega)$. Finally, the \autoref{ssec:sun dual semigroup} is dedicated to the proof of our main result, \autoref{thm:main}, whereas we prove \autoref{cor:main} in \autoref{ssec:L1-stokes}.

\newpage 

\section{Preliminaries}
\label{sec:prelim}

In this section we recall some notations and definitions, as well as some useful tools we need later. Throughout this whole work, $d\geqslant 2$ is a given fixed integer. We denote by $\Omega$ an open and bounded subset of $\R^d$.

\subsection{Sun Dual semigroups}
\label{ssec:sun dual}

\

Let $A \colon D(A) \subset X \to X$ be the  generator of a $\rC_0$-semigroup $(T(t))_{t \geqslant 0}$ on a Banach space $X$. The dual space of $X$ is denoted by $X^\prime$ and the dual operator of $A$ by $A^\prime$, and the dual operators of $T(t)$ by $T(t)'$. 
We define its \emph{sun-dual space} as the largest subspace of its dual space $X'$, where its dual semigroup $(T(t)')_{t \geqslant 0}$ is strongly continuous, i.e.
\begin{equation*}
	X^{\odot} := \{ x' \in X^\prime \colon \lim_{t \downarrow 0} \| T(t)' x' - x' \|_{X^\prime} = 0 \}
\end{equation*}
and call the semigroup formed by the restricted dual operators
\begin{equation*}
	T^{\odot}(t) := T(t)'|_{X^{\odot}}, \quad t \geqslant 0
\end{equation*}
the \emph{sun-dual semigroup}. It is by construction strongly continuous on $X^{\odot}$ and its generator is $A^{\odot} \colon D(A^{\odot}) \subset X^{\odot} \to X^{\odot}$, the part of the adjoint operator $A'$ in $X^{\odot}$, i.e.
\begin{equation*}
	A^{\odot} x' = A' x', \quad D(A^{\odot}) = \{ x' \in D(A') \colon A' x' \in X^{\odot} \},
\end{equation*} 
see \cite[Section II.2.6]{EN:00}. Moreover, we have the following characterization.

\begin{lem}\label{lem:sun dual closure}
	Let $A \colon D(A) \subset X \to X$ be the generator of a $\rC_0$-semigroup $(T(t))_{t \geqslant 0}$ on a Banach space $X$.
	Then $X^{\odot} = \overline{D(A^\prime)}^{X^\prime}$.
\end{lem}

Let us now summarize the semigroup properties of the sun-dual operator. We refer the reader to \cite[Section II.2.6]{EN:00}.

\begin{prop}\label{prop:sun dual}
	Let $A \colon D(A) \subset X \to X$ be the generator of a $\rC_0$-semigroup $(T(t))_{t \geqslant 0}$ on a Banach space $X$. Then $A^{\odot}$ generates a $\rC_0$-semigroup $(T^{\odot}(t))_{t \geqslant 0}$ on a Banach space $X^{\odot}$. If the semigroup $(T(t))_{t \geqslant 0}$ is analytic/ bounded/ compact, then the semigroup $(T^{\odot}(t))_{t \geqslant 0}$ is so. If so, the angles of analyticity coincide. 
\end{prop}

By repeating the sun-dual construction we define the \emph{sun bi-dual space} by 
\begin{equation*}
	X^{\odot \odot} := (X^{\odot})^{\odot},
\end{equation*}
where the sun-dual spaces of $X^{\odot}$ are taken with respect to the sun-dual semigroup $(T^{\odot}(t))_{t\geqslant 0}$. We call a space $X$ \emph{sun-dual reflexive} if $X^{\odot \odot} \cong X$. 
A classic result of Phillips, see \cite{Phi:55}, provides a sufficient criterion for sun-dual reflexivity.

\begin{lem}\label{lem:sun dual reflexiv}
	Let $A \colon D(A) \subset X \to X$ be the generator of a $\rC_0$-semigroup $(T(t))_{t \geqslant 0}$ on a Banach space $X$. Assume that $A$ has compact resolvent, then $X$ is sun-dual reflexive. 
\end{lem}

\subsection{Radon measures, function of bounded variations}
\label{ssec:Radon Meas. and BV}

\

The space of $\C$-valued Radon measures on $\bar\Omega$, where $\Omega\subset \R^d$ is an open set with $\rC^1$ boundary, $d\geqslant 2$, is denoted $M(\bar{\Omega},\C)$ and can be identified via the  Riesz-Markov-Kakutani representation theorem  
\begin{align*}
    M(\bar{\Omega},\C) = \rC(\bar\Omega,\C)',
\end{align*}
where $\rC(\bar\Omega,\C)$ denotes the vector space of $\C$-valued continuous functions on $\Omega$, up to the boundary.

\smallskip 
\

In that regard, $M({\Omega},\C): =\{ \mu|_\Omega,\ \mu\in M(\bar{\Omega},\C)\}$, can be canonically identified via the dual identity
\begin{align*}
    M({\Omega},\C) = \rC_0(\Omega,\C)',
\end{align*}
where $\rC_0(\Omega,\C):=\{ u \in \rC(\bar{\Omega},\C) \colon  \ u|_{\partial \Omega} = 0 \}$.

\smallskip 
\

The space of functions of bounded variations is defined to be
\begin{align*}
    \mathrm{BV}(\Omega):=\{ u\in \rL^1({\Omega},\C): \nabla u \in M({\Omega},\C^d) \}.
\end{align*}
This is a Banach space when endowed with its natural norm. By the Radon--Nikodym Theorem, we have the canonical isometric embeddings
\begin{align*}
    \rL^1(\Omega,\C^d)\hookrightarrow M(\Omega,\C^d),\quad\text{ and }\quad \rW^{1,1}(\Omega,\C)\hookrightarrow \mathrm{BV}(\Omega,\C).
\end{align*}

Furthermore, the trace operator
\begin{align*}
    \mathrm{BV}(\Omega,\C) &\longrightarrow \rL^1(\partial\Omega,\C)\\
    u\quad&\longmapsto\quad u|_{\partial\Omega}
\end{align*}
is well-defined, bounded and onto, since the trace operator from $\rW^{1,1}(\Omega)\subset\mathrm{BV}(\Omega)\longrightarrow \rL^1(\partial\Omega)$ is onto.

The range of the gradient is denoted by space $$\nabla\mathrm{BV}(\Omega)=\{ \nabla u,\  u \in \mathrm{BV}({\Omega},\C) \}\subset M(\Omega,\C^d).$$ The extended gradient range is  $$\nabla\mathrm{BV}(\bar\Omega)=\{ \nabla u|_{\Omega} - \nu\, u|_{\partial\Omega}\,\mathrm{d}\sigma,\  u \in \mathrm{BV}({\Omega},\C) \}\subset  \nabla\mathrm{BV}(\Omega,\C) + \nu M(\partial\Omega,\C)\subset M(\bar\Omega,\C^d),$$ and acts on $\rC(\bar\Omega,\C^d)$ by the formula
\begin{align*}
    \langle \nabla u|_{\Omega} - \nu\, u|_{\partial\Omega}\mathrm{d}\sigma,\, \varphi\rangle &= \langle \nabla u|_{\Omega},\, \varphi\rangle - \int_{\partial\Omega} \varphi \cdot\nu\, u|_{\partial\Omega} \, \mathrm{d} \sigma  ,\quad\forall \varphi\in\rC(\bar\Omega,\C^d)\\
    &=-\int_{\Omega} \div \varphi(x)\, u(x) \, \mathrm{d} x  ,\quad\quad\quad\quad\forall \varphi\in\rC^1(\bar\Omega,\C^d).
\end{align*}
Here, $\nu\in\rC(\partial\Omega,\R^d)$ stands for the outward unit normal of $\Omega$ on the boundary.

\subsection{The Stokes semigroup on $\rL^p_{\sigma,n}(\Omega)$ and $\rC_{\sigma,0}(\Omega)$}
\label{ssec:stokes lp}

\

In this section we recall the $\rL^p$ and $\rC$-theory of the Stokes operator. First, we define the spaces as
\begin{align*}
	\rL^p_{\sigma,n}(\Omega)
	&:= \{ u \in \rL^p(\Omega,\C^d) \colon \div u = 0, \ u \cdot \nu |_{\partial \Omega} = 0 \}, \\
	\rC_{\sigma,0}(\Omega)
	&:= \{ u \in \rC(\bar{\Omega},\C^d) \colon \div u = 0, \ u|_{\partial \Omega} = 0 \}, \\
	\rW^{1,1}_{\sigma,0}(\Omega)
	&:= \{ u \in \rW^{1,1}(\Omega,\C^d) \colon \div u = 0, \ u|_{\partial \Omega} = 0 \}, 
\end{align*}
for $p \in [1,\infty]$,
and more generally for  Sobolev-Slobodeckij space, Bessel potential spaces and Besov spaces $X^{s,p} \in \{ \rW^{s,p}, \rH^{s,p}, \rB_{pq}^s \}$
with $s \in \R$ and $p \in (1,\infty)$, $q\in[1,\infty]$, we have
\begin{equation*}
	X^{s,p}_{\sigma,n}(\Omega)
	:= \{ u \in X^{s,p}(\Omega,\C^d) \colon \div u = 0, \ u \cdot \nu |_{\partial \Omega} = 0 \},
\end{equation*}
if $s \in (-1+\tfrac{1}{p},\tfrac{1}{p})$ and
\begin{align*}
	X^{s,p}_{\sigma,0}(\Omega)
	:= \{ u \in X^{s,p}(\Omega,\C^d) \colon \div u = 0, \ u|_{\partial \Omega} = 0 \},
\end{align*}
if $s \in (\tfrac{1}{p},1+\tfrac{1}{p})$. If $s\in(1+1/p,2+1/p)$, we set $X^{s,p}_{\sigma,D}(\Omega):=X^{s-1,p}_{\sigma,0}(\Omega)\cap X^{s,p}(\Omega,\C^d)$.
We recall the following facts about these spaces from \cite[Chap.~4]{BG:25}.

\begin{lem}\label{lem:density}
	Let $p \in (1,\infty)$ and $s \in (-1+\tfrac{1}{p},\tfrac{1}{p})$ and $\Omega$ be a bounded Lipschitz domain. Then $\rC^\infty_{\sigma,c}(\Omega)$ is strongly dense in $X^{s,p}_{\sigma,n}(\Omega), \rC_{\sigma,0}(\Omega),
	X^{s+1,p}_{\sigma,0}(\Omega)$ and $\rL^1_{\sigma,n}(\Omega)$, $\rW^{1,1}_{\sigma,0}(\Omega)$. 
\end{lem}

For $p \in (1,\infty)$ and $s \in (-2+\tfrac{1}{p},-1+\tfrac{1}{p})$ we define
\begin{equation*}
	\rH^{s,p}_{\sigma}(\Omega)
	:= (\rH^{-s,p'}_{\sigma,0}(\Omega))'
\end{equation*}
with $\tfrac{1}{p} + \tfrac{1}{p'} = 1$. 
We recall the following result about the Stokes operator with no-slip boundary conditions on $\rL^p_{\sigma,n}(\Omega)$.

\begin{prop}\label{prop:stokes Lp}
	Let $\Omega \subset \R^d$ be a bounded $\rC^{1,\alpha}$-domain with $\alpha > 0$. Let $p\in(1,\infty)$, the Stokes operator $A$ on $\rL^p_{\sigma,n}(\Omega)$ is invertible and admits a bounded $\rH^\infty$-calculus of angle $0$ on 
	\begin{equation*}
		D_p(A^{\frac{s}{2}})
		= 
		\begin{cases}
			\rH^{s,p}_{\sigma}(\Omega), &\text{ if } s \in (-2+\tfrac{1}{p},-1+\tfrac{1}{p}),\\
			\rH^{s,p}_{\sigma,n}(\Omega), &\text{ if } s \in (-1+\tfrac{1}{p},\tfrac{1}{p}),\\
			\rH^{s,p}_{\sigma,0}(\Omega), &\text{ if } s \in (\tfrac{1}{p},1+\tfrac{1}{p}),\\
			\rH^{s,p}_{\sigma,D}(\Omega), &\text{ if } s \in (1+\tfrac{1}{p},1+\alpha+\tfrac{1}{p}) ,
		\end{cases}
	\end{equation*}
accompanied by the underlying isomorphisms. 
\end{prop}
\begin{proof}
	By \cite[Chap.~6]{BG:25} we already know the cases $s > -1 + \tfrac{1}{p}$. The other case holds by duality and self-adjointness on $\rL^2_{\sigma,n}(\Omega)$.
\end{proof}

Next, we recall the breakthrough result of Abe--Giga \cite{AG:12} and its refinements by Abe--Giga--Hieber \cite{AGH:13}, Geng--Shen \cite{GS:25} and Breit--Gaudin \cite{BG:25}. It will play a crucial role in our proof. 

\begin{thm}\label{thm:stokes Linfty}
	Let $\Omega \subset \R^d$ be a bounded $\rC^{1,\alpha}$-domain with $\alpha\in(0,1)$ and small $(1,\alpha)$-H\"{o}lder constant\footnote{This means, that up to increase the number of local charts, if $\varphi\,:\,\R^{d-1}\longrightarrow\R$ is a local description of the boundary of $\partial\Omega$, the quantity $\lVert \nabla'\varphi\rVert_{\rC^{0,\alpha}(\R^{d-1})}$ can be chosen small enough.}. Further, let $\theta \in [0,\pi)$. For all $\lambda \in \Sigma_{\theta}$ and all $f \in \rL^\infty_{\sigma,n}(\Omega)$ there exists a unique solution $(u,p) \in \rC^{1,\alpha}(\bar{\Omega},\C^d) \times \rC^{\alpha}_{\mathrm{mean-free}}(\bar{\Omega})$ of 
	\begin{equation*}
		\left\{
			\begin{aligned}
				\lambda u - \Delta u + \nabla p &= f, \\
				\div u &= 0, \\
				u|_{\partial \Omega} &= 0, 
			\end{aligned}
		\right.
	\end{equation*}
	satisfying the estimate
	\begin{equation*}
		(1+|\lambda|) \| u \|_{\rL^\infty(\Omega)}
		+ 
		(1+|\lambda|)^{\frac{1-\alpha}{2}} 
		\| \nabla u, p \|_{\rC^\alpha(\Omega)}
		\leqslant C_{\Omega} \cdot \| f \|_{\rL^\infty(\Omega)} .
	\end{equation*}
	Furthermore, the resolvent of the Stokes operator $A$ is compact. In particular, the Stokes operator $A$ is such that $-A$ generates a bounded, compact and analytic semigroup of angle $\frac{\pi}{2}$ on $\rL^\infty_{\sigma,n}(\Omega)$. 
    Furthermore, its restriction to $\rC_{\sigma,0}(\Omega)$ is a strongly continuous analytic $\rC_0$-semigroup with the same analyticity angle.
\end{thm}

We finish this section with a partial $\rL^1$-result of Breit-Gaudin \cite{BG:25} and several consequences. Since the resolvent is compact  and the operator has its spectrum contained in the positive real line and is invertible, one even obtains an additional exponential decay.

\begin{lem}\label{lem:almost L1}
	Let $\Omega \subset \R^d$ be a bounded $\rC^{1,\alpha}$-domain with $\alpha > 0$. Let $\theta \in (0,\tfrac{\pi}{2})$. For all $z \in \Sigma_{\tfrac{\pi}{2}-\theta}$ and all $u \in \rL^1_{\sigma,n}(\Omega)$ we have for all $s \in (0,2+\alpha)$
	\begin{equation*}
		\mathrm{e}^{- z A} u \in \rW^{s,1}(\Omega,\C^d) \cap \rL^1_{\sigma,n}(\Omega) .
	\end{equation*}
	Furthermore, for all $s \in (0,2+\alpha)$, we have the estimate
	\begin{equation*}
		\| \mathrm{e}^{- z A} u \|_{\rW^{s,1}(\Omega)}
		\leqslant \frac{C \mathrm{e}^{-c\Re(z)}}{|z|^{\frac{s}{2}}} \| u \|_{\rL^1(\Omega)} .
	\end{equation*}
\end{lem}

In particular, we already know that $\mathrm{e}^{- z A}$ is well-defined and uniquely determined on $\rL^1_{\sigma,n}(\Omega)$. The near equivalence between holomorphic semigroup and resolvent estimates leads to the following result

\begin{cor}Let $\Omega \subset \R^d$ be a bounded $\rC^{1,\alpha}$-domain with $\alpha > 0$. Let $\theta \in (0,\pi)$. For all $\lambda \in \Sigma_{\theta}$ and all $f \in \rL^1_{\sigma,n}(\Omega)$ there exists a unique solution $(u,p) \in \rW^{2^-,1}({\Omega},\C^d) \times \rW^{1^-,1}_{\mathrm{mean-free}}({\Omega})$ of 
	\begin{equation*}
		\left\{
			\begin{aligned}
				\lambda u - \Delta u + \nabla p &= f, \\
				\div u &= 0, \\
				u|_{\partial \Omega} &= 0, 
			\end{aligned}
		\right.
	\end{equation*}
	satisfying the estimate, for any $\varepsilon\in(0,1]$,
	\begin{equation*}
		(1+|\lambda|)^{1-\frac{\varepsilon}{2}} \| u \|_{\rL^1(\Omega)}+
		(1+|\lambda|)^{\frac{\varepsilon}{2}} 
		\| \nabla u\|_{\rW^{1-\varepsilon,1}(\Omega)} + (1+|\lambda|)^{-\frac{\varepsilon}{2}}\| p\|_{\rW^{1-\varepsilon,1}(\Omega)}
		\leqslant C_{\Omega}^{\varepsilon} \cdot \| f \|_{\rL^1(\Omega)} .
	\end{equation*}
\end{cor}

\begin{proof} Since the semigroup is holomorphic of angle $\frac{\pi}{2}$, up to multiplying $A$ by $e^{i\pm\phi}$, $\phi\in[0,\pi/2)$, and since the semigroup has exponential decay, up to a small translation by the identity, we can assume w.l.o.g. $\Re(\lambda) \geqslant 0$. Let $f\in \rL^2_{\sigma,n}(\Omega)\subset \rL^1_{\sigma,n}(\Omega)$, then for $\lambda\in\C$ such that $\Re (\lambda) \geqslant 0$, it holds
\begin{align*}
    R(\lambda,-A)f =\int^{\infty}_0 \mathrm{e}^{-\lambda t}\mathrm{e}^{-t A}f \,\mathrm{d}t.
\end{align*}

Notice that for $u:=R(\lambda,-A)f\in\rH^{1,2}_{0,\sigma}(\Omega)$, there exists a unique $p\in\rL^2(\Omega)$, mean-free, such that
\begin{equation*}
		\left\{
			\begin{aligned}
				\lambda u - \Delta u + \nabla p &= f, \\
				\div u &= 0, \\
				u|_{\partial \Omega} &= 0.
			\end{aligned}
		\right.
\end{equation*}

Let $\varepsilon\in(0,1]$, by the embedding $\rW^{\varepsilon,1}(\Omega)\hookrightarrow\rL^1(\Omega)$ and \autoref{lem:almost L1},
\begin{align*}
    \lVert u\rVert_{\rL^1(\Omega)} &\leqslant \int_{0}^{\infty}\lVert \mathrm{e}^{-t\lambda }\mathrm{e}^{-t A}f\rVert_{\rL^1(\Omega)} \, \mathrm{d}t\leqslant C^{\varepsilon}_{\Omega} \int_{0}^{\infty}\lVert \mathrm{e}^{-t\lambda }\mathrm{e}^{-t A}f\rVert_{\rW^{\varepsilon,1}(\Omega)}\,  \mathrm{d}t\\
    &\leqslant C^{\varepsilon}_{\Omega} \int_{0}^{\infty} \mathrm{e}^{-t \Re(\lambda) } \frac{e^{-ct}}{t^{\varepsilon/2}}\lVert f\rVert_{\rL^{1}(\Omega)} \, \mathrm{d}t \leqslant C^{\varepsilon}_{\Omega} \frac{1}{(1 + \Re (\lambda))^{1-\frac{\varepsilon}{2}}}\lVert f\rVert_{\rL^{1}(\Omega)}.
\end{align*}
Similarly,
\begin{align*}
    \lVert u\rVert_{\rW^{2-\varepsilon,1}(\Omega)} &\leqslant C^{\varepsilon}_{\Omega} \frac{1}{(1 + \Re (\lambda))^{\frac{\varepsilon}{2}}}\lVert f\rVert_{\rL^{1}(\Omega)}.
\end{align*}
Finally, we have obtained
\begin{align*}
    (1+|\lambda|)^{1-\frac{\varepsilon}{2}} \| u \|_{\rL^1(\Omega)}+
		(1+|\lambda|)^{\frac{\varepsilon}{2}} 
		\| \nabla u\|_{\rW^{1-\varepsilon,1}(\Omega)}
		\leqslant C_{\Omega}^{\varepsilon} \cdot \| f \|_{\rL^1(\Omega)} .
\end{align*}

Now in order to estimate the pressure $p$, we use the equation and the previous estimates and we obtain
\begin{align*}
    \lVert \nabla p\rVert_{\rW^{-\varepsilon,1}(\Omega)} &\leqslant C^{\varepsilon}_{\Omega} (|\lambda|\| u\|_{\rW^{-\varepsilon,1}(\Omega)}+\| \nabla u\|_{\rW^{1-\varepsilon,1}(\Omega)}+ \lVert f\rVert_{\rW^{-\varepsilon,1}(\Omega)})\\
     &\leqslant C^{\varepsilon}_{\Omega} (|\lambda|^{\varepsilon/2}+1)\lVert f\rVert_{\rL^{1}(\Omega)}.
\end{align*}
This finishes the proof, since $\rL^2_{\sigma,n}(\Omega)\subset \rL^1_{\sigma,n}(\Omega)$ is strongly dense.
\end{proof}

The structure of the proof allows us to consider the closure of the $\rL^2_{\sigma,n}(\Omega)$--operator to define the Stokes operator on $\rL^1_{\sigma,n}(\Omega)$ with an explicit description of the domain as an immediate consequence.

\begin{cor}\label{cor:Stokes Operator on L1}Let $\Omega \subset \R^d$ be a bounded $\rC^{1,\alpha}$-domain with $\alpha > 0$. The Stokes operator ${A} \,\colon\, D_1({A}) \subset \rL^1_{\sigma,n}(\Omega) \to \rL^1_{\sigma,n}(\Omega)$  with no-slip boundary conditions defined by
	\begin{align*}
		{A} u &= -\Delta u+\nabla p,\quad  \text{ with domain }\\
		 D_1({A}) &= \{ u \in  \rL^1_{\sigma,n}\cap\rW^{1,1}_0\cap\rW^{2^-,1}(\Omega,\C^d)\, \colon\, \exists p\in\rW^{1^-,1}(\Omega), \text{ s.t. }\Delta u-\nabla p \in \rL^1_{\sigma,n}(\Omega)  \},
	\end{align*}
    yields a densely defined closed and invertible operator on $\rL^1_{\sigma,n}(\Omega)$. Furthermore, there exists $a_\ast>0$ such that $\sigma(A)\subset[a_\ast,\infty)$. 
\end{cor}

\begin{rem}
    Since the Stokes operator $A$ has compact resolvent on $\rL^p_{\sigma,n}(\Omega)$ for $p \in [1,\infty)$, due to the boundedness of $\Omega$ and self-adjointness, their spectra consist only of discrete eigenvalues and they coincide, i.e. the spectrum of the Stokes operator $A$ is independent of $p \in [1,\infty)$. Since the Stokes operator $A$ on  $\rL^2_{\sigma,n}(\Omega)$ is self-adjoint all eigenvalues are real. This implies that the spectral gap $a_\ast>0$ from \autoref{cor:Stokes Operator on L1} is the smallest eigenvalue of $A$ and can be determined from the $\rL^2$-theory of the Stokes operator.
\end{rem}

\begin{rem}\label{SchaftingenRefinement}
    Actually, there is a slightly better amount of integrability available for the gradient of the solution. It holds
    \begin{align*}
        D_1({A}) \subset\rW^{1,\frac{d}{d-1}}_{\sigma,0}(\Omega),
    \end{align*}
    which is unreachable via the description above and standard Sobolev embeddings.
    Indeed, any $f\in\rL^1_{\sigma,n}(\Omega)$ can be extended by $0$ to the whole $\R^d$ as a divergence-free vector field, so that by \cite[Corollary 1.4]{VanScha:04},
    \begin{align*}
        \Bigg|\int_{\Omega} f\cdot\varphi \Bigg| = \Bigg|\int_{\R^d} f\cdot\varphi \Bigg| \leqslant C \cdot \| f\|_{\rL^1(\Omega)} \cdot \| \nabla \varphi\|_{\rL^{d}(\Omega)},\quad\forall \varphi\in\rC_{c}^\infty(\Omega,\C^d).
    \end{align*}
    Hence, $f\in \rL^1_{\sigma,n}(\Omega) \hookrightarrow \rW^{1,d}_{0}(\Omega,\C^d)' = \rW^{-1,\frac{d}{d-1}}(\Omega,\C^d)$. A duality argument, writing $f=\div F$ for $F\in\rL^{\frac{d}{d-1}}(\Omega,\C^{d\times d})$, and  $u = A^{-1}\mathbb{P} \div F = A^{-1/2} (\nabla A^{-1/2})'  F $ in combination with \autoref{prop:stokes Lp}, yields $u\in \rW^{1,\frac{d}{d-1}}_{\sigma,0}(\Omega)$. By linearity, one can recover $p\in\rL^{\frac{d}{d-1}}(\Omega)$.
\end{rem}

\section{The $\rL^1$-Stokes semigroup}
\label{sec:main part}

In this section we study the Stokes operator on $\rL^1(\Omega,\C^d)$. We start with a discussion of the failure of the Helmholtz decomposition in $\rL^1(\Omega,\C^d)$.

\subsection{The Failure of the Helmholtz decomposition in $\rL^1(\Omega,\C^d)$ revisited}
\label{ssec:L1 no helmholtz}

\

It is a well-known fact, see \cite{BB:02}, that the Helmholtz decomposition fails in $\rL^1(\Omega,\C^d)$, i.e.
\begin{equation}
	\rL^1(\Omega,\C^d) \not = \rL^1_{\sigma,n}(\Omega) + \nabla \rW^{1,1}(\Omega) . \label{eq:no L1 helmholtz decomposition}
\end{equation}
In this section we will analyse this failure in more detail. We first note that both summands are closed and disjoint. More precisely, we obtain the following result. 

\begin{lem}\label{lem:L1,gradients subspace M}
	Let $\Omega \subset \R^d$ be a bounded Lipschitz domain. Then $\rL^1_{\sigma,n}(\Omega), \nabla \rW^{1,1}(\Omega)$ and $\nabla \mathrm{BV}(\Omega)$ are proper closed subspaces of $M(\Omega,\C^d)$ and, if additionally $\Omega$ has a $\rC^{1,\alpha}$ boundary, $\alpha>0$, one has
	\begin{equation*}
		\rL^1_{\sigma,n}(\Omega) \cap \nabla \mathrm{BV}(\Omega)
		= 
		\rL^1_{\sigma,n}(\Omega) \cap \nabla \rW^{1,1}(\Omega)
		= \{ 0 \} .
	\end{equation*}
\end{lem}
\begin{proof}
	Note that $\rL^1_{\sigma,n}(\Omega)$ is a proper closed subspace of $\rL^1(\Omega,\C^d)$ and $\rL^1(\Omega,\C^d)$ is a proper closed subspace of $M(\Omega,\C^d)$ which implies the first claim. For the two other results, we obtain using the Poincar\'{e}-Wirtinger inequality that 
	\begin{equation*}
		\left\| f - \fint_\Omega f \right\|_{\rL^1(\Omega)}
		\leqslant C_{\Omega} \cdot \| \nabla f \|_{M(\Omega)}
	\end{equation*}
	and conclude that the gradient $\nabla \colon \rW^{1,1}(\Omega) \subset \rL^1(\Omega) \to \rL^1(\Omega,\C^d)$ and $\nabla \colon \mathrm{BV}(\Omega) \subset M(\Omega) \to M(\Omega,\C^d)$ have closed ranges. 

    Now, let $u\in \rL^1_{\sigma,n}(\Omega) \cap \nabla \mathrm{BV}(\Omega)$. Then $u=\nabla p\in\rL^1_{\sigma,n}(\Omega)$, such that $\nabla p\cdot \nu =u\cdot\nu =0$ on $\partial\Omega$ and $\Delta p = \div u =0$. By Sobolev embeddings, the Neumann Laplace problem admits at most one solution (up to a constant), this implies $u=0$.
\end{proof}

Even if both summands are closed subspaces the direct sum is not closed in $\rL^1(\Omega,\C^d)$. As the next result shows this is exactly what makes the Helmholtz decomposition fail in $\rL^1(\Omega,\C^d)$.

\begin{lem}\label{lem:helmholtz l1 dense}
	Let $\Omega \subset \R^d$ be a bounded Lipschitz domain. Then the sum 
	\begin{equation*}
		\rL^1_{\sigma,n}(\Omega) + \nabla \rW^{1,1}(\Omega)
	\end{equation*}	
	is dense in $\rL^1(\Omega,\C^d)$.
\end{lem}
\begin{proof}
	Consider $u \in \rL^1(\Omega,\C^d)$. Since $\Omega \subset \R^d$ is bounded, there exists a sequence $(u_k)_{k \in \N} \subset \rL^2(\Omega,\C^d)$ which converges strongly to $u$ in $\rL^1(\Omega,\C^d)$. Furthermore, the Helmholtz decomposition in $\rL^2(\Omega,\C^d)$ implies the existence of $v_k \in \rL^2_{\sigma,n}(\Omega)$ and $p_k \in \rW^{1,2}(\Omega)$ such that
	\begin{equation*}
		u_k = v_k + \nabla p_k \in \rL^2_{\sigma,n}(\Omega) \oplus \nabla \rW^{1,2}(\Omega) \hookrightarrow \rL^1_{\sigma,n}(\Omega) + \nabla \rW^{1,1}(\Omega). \qedhere 
	\end{equation*}
\end{proof}

\begin{prop}Let $\Omega\subset\R^d$ be a bounded Lipschitz domain. Then the Helmholtz projection $\mathbb{P}$ on $\rC(\bar{\Omega},\C^d)$, with domain
\begin{align*}
    D_0(\mathbb{P})=\{ u\in\rC(\bar{\Omega},\C^d):\mathbb{P}u\in\rC(\bar{\Omega},\C^d)\}
\end{align*}
defines a closed projection operator such that
\begin{enumerate}[(i)]
    \item $\mathrm{rg}_0(\mathbb{P})=\rC_{\sigma,n}(\Omega)\subset D_0(\mathbb{P})$;
    \item $\ker_0(\mathbb{P})=\nabla \rC^1(\bar{\Omega})\subset D_0(\mathbb{P})$;
    \item $D_0(\mathbb{P}) = \ker_0(\mathbb{P})\oplus \mathrm{rg}_0(\mathbb{P})$;
    \item $\mathbb{P}^2=\mathbb{P}$ on $\mathrm{rg}_0(\mathbb{P})$.
\end{enumerate}
Furthermore, if $\Omega$ is a bounded $\rC^{1,\alpha}$-domain, $\alpha>0$, it is densely defined.
\end{prop}

\begin{proof} Well-definedness, closedness as well as validity of points (i), (ii), (iii) and (iv) are direct consequences from the fact that, since $\Omega$ is bounded, one has  the continuous embedding $\rC(\bar{\Omega},\C^d)\hookrightarrow \rL^2(\Omega,\C^d)$.

It remains to show that it is densely defined when $\Omega$ is a bounded $\rC^{1,\alpha}$-domain. It is sufficient to show that $\rC^{\beta}(\bar\Omega,\C^d)\subset D_0(\mathbb{P})$, $0<\beta<\alpha$. Since $\mathbb{P}u=u+\nabla(-\Delta_N)^{-1}\div u$, it suffices to check that $\nabla(-\Delta_N)^{-1}\div u\in\rC^{\beta}(\bar\Omega,\C^d)$ whenever $u\in\rC^{\beta}(\bar\Omega,\C^d)$. Here, $\Delta_N$ stands for the Laplace operator subject to Neumann boundary conditions. However,  $\div\,:\,\rC^{\beta}(\bar\Omega,\C^d)=\rB^{\beta}_{\infty,\infty}(\Omega,\C^d)\longrightarrow \rB^{\beta-1}_{\infty,\infty,0}(\Omega,\C)$ and, up to considering quotient by constant distributions, $(-\Delta_N)^{-1}\,:\,\rB^{\beta-1}_{\infty,\infty,0}(\Omega,\C)\longrightarrow\rB^{\beta+1}_{\infty,\infty}(\Omega,\C)$ are well-defined and bounded. The latter is true since $\Omega$ is bounded with $\rC^{1,\alpha}$-boundary. This ends the proof.
\end{proof}

By duality, we obtain the following straightforward consequence.

\begin{cor}\label{cor:Helmholtz Proj on Radon Meas./L1}Let $\Omega\subset\R^d$ be a bounded $\rC^{1,\alpha}$-domain, $\alpha>0$. Then the dual Helmholtz projection $\mathbb{P}'$ on $M(\bar{\Omega},\C^d)$, with dual domain
\begin{align*}
    D_{M}(\mathbb{P}')=\{ u\in M(\bar{\Omega},\C^d):\mathbb{P}'u\in M(\bar{\Omega},\C^d)\}
\end{align*}
defines a closed projection operator such that
\begin{enumerate}[(i)]
    \item $\mathrm{rg}_M(\mathbb{P}')=M_{\sigma,n}(\bar\Omega)\subset D_M(\mathbb{P}')$;
    \item $\ker_M(\mathbb{P}')=\nabla \mathrm{BV}({\Omega})\oplus\nu M(\partial\Omega)\subset D_M(\mathbb{P}')$;
    \item $D_M(\mathbb{P}') = \ker_M(\mathbb{P}')\oplus \mathrm{rg}_M(\mathbb{P}')$;
    \item $(\mathbb{P}')^2=\mathbb{P}'$ on $\mathrm{rg}_M(\mathbb{P}')$.
    \item $\mathbb{P}' = \mathbb{P}$ on $\rL^2(\Omega,\C^d)\subset M(\bar{\Omega},\C^d)$.
\end{enumerate}
Furthermore,  its natural part in $\rL^1(\Omega,\C^d)\subset M(\bar{\Omega},\C^d)$ yields a closed densely defined operator with the corresponding properties.
\end{cor}

The construction of the Leray projection as an unbounded operator on the spaces of continuous or integrable functions is the best type of result one can expect.

Although standard --the next result seems to be very well-known--, we provide a proof for the reader's convenience that the Helmholtz decomposition systematically fails on $\rL^1(\Omega)$, $\rC(\bar\Omega)$ and $M(\bar\Omega)$.

\begin{prop}\label{prop:L1 Helmholtz always fails} Let $\Omega$ be a bounded $\rC^{1,\alpha}$-domain, $\alpha>0$. It holds
\begin{align*}
    \rL^1(\Omega,\C^d) &\not =  \rL^1_{\sigma,n}(\Omega) \oplus \nabla \rW^{1,1}(\Omega),\\
    \rC(\bar\Omega,\C^d) &\not =  \rC_{\sigma,n}(\Omega) \oplus \nabla \rC^{1}(\bar\Omega),\\
    M(\bar\Omega,\C^d) &\not =  M_{\sigma,n}(\bar\Omega)\oplus [\nabla \mathrm{BV}(\Omega)\oplus \nu M(\partial\Omega)].
\end{align*}
\end{prop}

\begin{proof}\textbf{Step 1:} Assume by contradiction that
\begin{align*}
    \rL^1(\Omega,\C^d)  =  \rL^1_{\sigma,n}(\Omega) \oplus \nabla \rW^{1,1}(\Omega).
\end{align*}
Let $x_0\in\Omega$, $r>0$, such that $\overline{\rB_{2r}(x_0)}\subset\Omega$ and consider $\rho= \mathbf{1}_{\rB_1(0)}$.
For any $\varepsilon>0$, we set $u_\varepsilon= \varepsilon^{-d}\rho(\frac{x-x_0}{\varepsilon})e_1 = \rho_\varepsilon(x-x_0)e_1$. Hence, provided $\varepsilon<r$, $u_\varepsilon$ has support in $\overline{\rB_{\varepsilon}(x_0)}\subset{\rB_{r}(x_0)}$, for $\varepsilon<r$. By assumption, taking the divergence of the assumed direct topological decomposition, there exists a unique $p_\varepsilon\in\rW^{1,1}(\Omega)$, since it can be chosen to be mean-free, such that $-\Delta p_\varepsilon =\div u_\varepsilon = \partial_{x_1}\rho_\varepsilon(\cdot-x_0)$ in $\Omega$ and $\nabla  p_\varepsilon\cdot\nu = u_\varepsilon\cdot \nu =0$ on $\partial\Omega$.
We set
\begin{align*}
    q_\varepsilon:= (-\Delta_{\R^d})^{-1}\partial_{x_1}\rho_\varepsilon(\cdot-x_0)
\end{align*}
and  $h_\varepsilon = p_\varepsilon-q_\varepsilon$ is such that $\Delta h_\varepsilon =0$ in $\Omega$ and satisfies $\nabla h_\varepsilon\cdot \nu =- \nabla q_\varepsilon \cdot \nu$ on $\partial\Omega$. For $x\in\partial\Omega$,
\begin{align*}
    |\nabla q_\varepsilon(x)| \leqslant \int_{\rB_r(x_0)} \frac{|\rho_\varepsilon(\cdot-x_0)|}{|x-y|^d} \mathrm{d} y \leqslant \frac{C}{r^d}.
\end{align*}
Thus $\nabla q_\varepsilon \cdot \nu \in\rL^\infty(\partial\Omega)\subset\rL^2(\partial\Omega)$, so that by Theorem~2 from \cite{JK:81}, it holds that $h_\varepsilon\in \rH^{1,2}(\Omega)\subset \rW^{1,1}(\Omega)$ with the following bound uniform with respect to $\varepsilon$:
\begin{align*}
    \lVert \nabla h_\varepsilon\rVert_{\rL^1(\Omega)} \leqslant C^r_{\Omega}.
\end{align*}
On the other hand,
\begin{align*}
    \lVert \nabla q_\varepsilon\rVert_{\rL^1(\Omega)}\geqslant \lVert R^2_1\rho_\varepsilon(\cdot-x_0)\rVert_{\rL^1(\Omega)}.
\end{align*}
Without loss of generality, we can assume $x_0=0$. For $R_1:=\partial_{x_1}(-\Delta_{\R^d})^{-1/2}$ to be the $d$-dimensional Riesz Transform with respect to the first variable, one has for almost every $x\in\Omega$:
\begin{align*}
    \partial_{x_1}q_\varepsilon(x)=R^2_1\rho_\varepsilon(x) = \frac{c_d}{\varepsilon^d}\int_{\rB_{\varepsilon}(0)} \frac{d(x_1-y_1)^2-|x-y|^2}{|x-y|^{d+2}}\mathrm{d} y.
\end{align*}
Now, if $|x|\geqslant 2\varepsilon$, since the integrand is harmonic it holds
\begin{align*}
    R^2_1\rho_\varepsilon(x) = c_d\frac{dx_1^2-|x|^2}{|x|^{d+2}}.
\end{align*}
Therefore, by change of variables in polar coordinates, we can conclude with
\begin{align*}
    \lVert R^2_1\rho_\varepsilon\rVert_{\rL^1(\Omega)} \geqslant c_d \int_{2\varepsilon<|y|<r} \left|\frac{dy_1^2-|y|^2}{|y|^{d+2}}\right|\mathrm{d} y = c_d \int_{2\varepsilon}^r \frac{\tau^{d-1}\tau^2}{\tau^{d+2}} \mathrm{d}\tau = c_d \ln\Big(\frac{r}{2\varepsilon}\Big)\xrightarrow[\varepsilon\rightarrow0_+]{}+\infty.
\end{align*}
\smallskip
\

\noindent\textbf{Step 2:} We check the remaining cases. If one assumes that the Helmholtz decomposition holds for Radon measures, then by \autoref{cor:Helmholtz Proj on Radon Meas./L1}, this decomposition restricts and extends as the bounded and topological Helmholtz decomposition on $\rL^1(\Omega,\C^d)$, which is a contradiction.

Similarly, assume that the Helmholtz decomposition holds for the space of continuous functions $\rC(\bar\Omega,\C^d)$. By duality and  \autoref{cor:Helmholtz Proj on Radon Meas./L1}, this implies that the Helmholtz decomposition holds for the space of Radon measures $M(\bar\Omega,\C^d)$, which leads to another contradiction.
\end{proof}

\subsection{The Sun-dual space}
\label{ssec:sun dual space}

\ 

In order to determine the sun-dual space of the Stokes operator $A$ on $X = \rC_{\sigma,0}(\Omega)$ we first characterize the dual space of $X$.

\begin{lem}\label{lem:X'}
	Let $\Omega \subset \R^d$ be a bounded $\rC^{1,\alpha}$-domain, $\alpha>0$. Then
	\begin{equation*}
		\rC_{\sigma,0}(\Omega)' = \nicefrac{M(\Omega,\C^d)}{\nabla \mathrm{BV}(\Omega)} .
	\end{equation*}
	Moreover, we have
	\begin{enumerate}[(i)]
		\item the isometric embedding $\nicefrac{\rL^1(\Omega,\C^d)}{\nabla \rW^{1,1}(\Omega)} \hookrightarrow \nicefrac{M(\Omega,\C^d)}{\nabla \mathrm{BV}(\Omega)}$;
		\item the canonical continuous embedding $\rL^1_{\sigma,n}(\Omega) \hookrightarrow \nicefrac{\rL^1(\Omega,\C^d)}{\nabla \rW^{1,1}(\Omega)}$ with dense range.
	\end{enumerate}
\end{lem}
\begin{proof}
	Trivially, we have the embedding $ \nicefrac{M(\Omega,\C^d)}{\nabla \mathrm{BV}(\Omega)} \hookrightarrow \rC_{\sigma,0}(\Omega)'$. For the opposite embedding, we consider $l \in \rC_{\sigma,0}(\Omega)'$ and obtain by the Hahn-Banach theorem that there exists $\mu \in M(\Omega,\C^d) = \rC_0(\Omega,\C^d)'$ such that
	\begin{equation*}
		l(\varphi)
		= \int_{\Omega} \varphi \,  \mathrm{d} \mu 
		= \langle \mu, \varphi \rangle 
	\end{equation*}
	for all $\varphi \in \rC_{\sigma,0}(\Omega)$. Assume now that there exists another $\mu'$ such that 
	\begin{equation*}
		l = \langle \mu, \cdot \rangle = \langle \mu', \cdot \rangle \qquad \text{ on } \rC_{\sigma,0}(\Omega) .
	\end{equation*} 
	Thus $\langle \mu - \mu', \cdot \rangle = 0$ on $\rC_{\sigma,0}(\Omega)$ and the de Rham Theorem, see Theorem 17' page 95 of \cite{DeRham:84}, guarantees the existence of $\pi \in \mathcal{D}'(\Omega)$ such that $\mu - \mu' = \nabla \pi$ in $\mathcal{D}'(\Omega,\C^d)$. Hence, $\nabla \pi \in M(\Omega,\C^d)$ and therefore $\pi \in \mathrm{BV}(\Omega)$, since $\Omega$ is bounded, which proves the opposite direction.
	
	\smallskip 
	
	For the embedding in (i) we note that by the Radon-Nikodym theorem 
	\begin{equation*}
		\rL^1_{\sigma,n}(\Omega)
		\hookrightarrow
		\nicefrac{\rL^1(\Omega,\C^d)}{\nabla \rW^{1,1}(\Omega)}
		\hookrightarrow 
		\nicefrac{M(\Omega,\C^d)}{\nabla \mathrm{BV}(\Omega)} .
	\end{equation*}
	Now, for $[u] \in 
	\nicefrac{\rL^1(\Omega,\C^d)}{\nabla \rW^{1,1}(\Omega)}
	\hookrightarrow 
	\nicefrac{M(\Omega,\C^d)}{\nabla \mathrm{BV}(\Omega)}$ we obtain
	\begin{equation*}
		\| [u] \|_{\nicefrac{\rL^1(\Omega,\C^d)}{\nabla \rW^{1,1}(\Omega)}}
		= \inf_{\pi \in \rW^{1,1}(\Omega)} \| u + \nabla \pi \|_{\rL^1(\Omega)}
		= \inf_{\pi \in \rW^{1,1}(\Omega)} \| u + \nabla \pi \|_{M(\Omega)}
		= \| [u] \|_{\nicefrac{M(\Omega,\C^d)}{\nabla \rW^{1,1}(\Omega)}} .
	\end{equation*}
	Fix $\varepsilon > 0$ and consider $\pi_{\varepsilon} \in \mathrm{BV}(\Omega)$ such that
	\begin{equation*}
		\| u + \nabla \pi_\varepsilon \|_{M(\Omega)}
		\leqslant  \| [u] \|_{\nicefrac{M(\Omega,\C^d)}{\nabla \mathrm{BV}(\Omega)}} + \varepsilon .
	\end{equation*}
	By \cite[Theorem 1]{AG:78} there exists a sequence $(\pi_{k,\varepsilon})_{k \in \N} \subset \rC^\infty(\bar{\Omega}) \subset \rW^{1,1}(\Omega)$ such that
	\begin{enumerate}[(i)]
		\item $\pi_{k,\varepsilon} \to \pi_{\varepsilon}$ in $\rL^1(\Omega)$ as $k \to + \infty$;
		\item total mass converges, i.e. $\int_{\Omega} | \nabla \pi_{k,\varepsilon} | \to \int_{\Omega} | \nabla \pi_{\varepsilon} |$ as $k \to + \infty$. 
	\end{enumerate}
	Therefore, for sufficiently large $k$ we obtain
	\begin{equation*}
		\| [u] \|_{\nicefrac{M(\Omega,\C^d)}{\nabla \rW^{1,1}(\Omega)}}
		\leqslant \| u + \nabla \pi_\varepsilon \|_{M(\Omega)} + \varepsilon
		\leqslant \| [u] \|_{\nicefrac{M(\Omega,\C^d)}{\nabla \mathrm{BV}(\Omega)}} + 2 \varepsilon . 
	\end{equation*}
	Since $\varepsilon > 0$ is arbitrary, we conclude 
	\begin{equation*}
		\| [u] \|_{\nicefrac{\rL^1(\Omega,\C^d)}{\nabla \rW^{1,1}(\Omega)}}
		= \| [u] \|_{\nicefrac{M(\Omega,\C^d)}{\nabla \rW^{1,1}(\Omega)}}
		\leqslant \| [u] \|_{\nicefrac{M(\Omega,\C^d)}{\nabla \mathrm{BV}(\Omega)}} .
	\end{equation*}
	Combining both inequalities, we conclude for $u \in \rL^1(\Omega,\C^d)$ that
	\begin{equation*}
		\| [u] \|_{\nicefrac{\rL^1(\Omega,\C^d)}{\nabla \rW^{1,1}(\Omega)}}
		= \| [u] \|_{\nicefrac{M(\Omega,\C^d)}{\nabla \mathrm{BV}(\Omega)}} .
	\end{equation*}
	
	\smallskip 
	
	Lastly, the canonical embedding in (ii) is a direct consequence from \autoref{lem:helmholtz l1 dense} 
	\begin{equation*}
		\rL^1_{\sigma,n}(\Omega)
		\hookrightarrow
		\nicefrac{\big[\overline{\rL^1_{\sigma,n}(\Omega) \oplus \nabla \rW^{1,1}(\Omega)}^{\| \cdot \|_{\rL^1}}\big]}{\nabla \rW^{1,1}(\Omega)}
		= \nicefrac{\rL^1(\Omega,\C^d)}{\nabla \rW^{1,1}(\Omega)} . \qedhere 
	\end{equation*}
\end{proof}

We conclude with the following fundamental corollary.

\begin{cor}\label{cor:closures}
	Let $\Omega \subset \R^d$ be a bounded $\rC^{1,\alpha}$-domain, $\alpha>0$. Then we have the following isometric identifications 
	\begin{equation*}
		\nicefrac{\rL^1(\Omega,\C^d)}{\nabla \rW^{1,1}(\Omega)}
		= \overline{\rL^p_{\sigma,n}(\Omega)}^{\| \cdot \|_{\nicefrac{M(\Omega,\C^d)}{\nabla \mathrm{BV}(\Omega)}}}
		= \overline{\rC_{\sigma,0}(\Omega)}^{\| \cdot \|_{\nicefrac{M(\Omega,\C^d)}{\nabla \mathrm{BV}(\Omega)}}}
		= \overline{\rC_{\sigma,c}^\infty(\Omega)}^{\| \cdot \|_{\nicefrac{M(\Omega,\C^d)}{\nabla \mathrm{BV}(\Omega)}}}
	\end{equation*}
	for all $p \in [1,\infty]$. 
\end{cor}

Now, we are sufficiently prepared to prove the main result of this subsection.
\begin{prop}\label{prop:sun dual stokes}
	Let $\Omega \subset \R^d$ be a bounded $\rC^{1,\alpha}$-domain.
	Consider 
	the Stokes operator $A$ on $X = \rC_{\sigma,0}(\Omega)$. Its sun dual space is given by
	\begin{equation*}
		X^{\odot} = \nicefrac{\rL^1(\Omega,\C^d)}{\nabla \rW^{1,1}(\Omega)} .
	\end{equation*} 
\end{prop}
\begin{proof}
	Let $p > d$ such that $\nicefrac{d}{p} < \varepsilon < 1$ with $\varepsilon > 0$ and $p < \infty$ large. We have the continuous embeddings
	\begin{equation*}
		D_{\varepsilon,p}(A) \hookrightarrow D_0(A) \hookrightarrow D_p(A),
	\end{equation*}
	where $D_{\varepsilon,p}(A)$ is the domain of the Stokes operator on $\rH^{\varepsilon,p}_{\sigma,0}(\Omega)$ for $\frac{1}{p} < \frac{d}{p} < \varepsilon < 1 < 1 + \frac{1}{p}$, and $D_0(A)$ the domain of the Stokes operator on $\rC_{\sigma,0}(\Omega)$ from \autoref{thm:stokes Linfty}.
	Using \autoref{lem:X'} we obtain by duality that
	\begin{equation}
		D_{p'}(A') \hookrightarrow D_{\nicefrac{M(\Omega,\C^d)}{\nabla\mathrm{BV}(\Omega)}}(A') \hookrightarrow D_{-\varepsilon,p'}(A') .
		\label{eq:domain embeddings}
	\end{equation}
	By \autoref{prop:stokes Lp}, this implies
	\begin{equation*}
		\rC^\infty_{\sigma,c}(\Omega)
		\hookrightarrow 
		\rW^{2,p'}_{0,\sigma}(\Omega)
		\hookrightarrow D_{\nicefrac{M(\Omega,\C^d)}{\nabla\mathrm{BV}(\Omega)}}(A')	
		\hookrightarrow 
		\rW^{1,p'}_{\sigma,0}(\Omega)
		\hookrightarrow 
		\rL^1_{\sigma,n}(\Omega) .
	\end{equation*}
	Now, using \autoref{lem:sun dual closure} and \autoref{cor:closures} we conclude
	\begin{equation*}
		X^{\odot}
		= \overline{D_{\nicefrac{M(\Omega,\C^d)}{\nabla\mathrm{BV}(\Omega)}}(A')}^{\| \cdot \|_{\nicefrac{M(\Omega,\C^d)}{\nabla\mathrm{BV}(\Omega)}}}	
		= 
		\overline{\rL^1_{\sigma,n}(\Omega)}^{\| \cdot \|_{\nicefrac{M(\Omega,\C^d)}{\nabla\mathrm{BV}(\Omega)}}} = \nicefrac{\rL^1(\Omega,\C^d)}{\nabla \rW^{1,1}(\Omega)} . \qedhere 
	\end{equation*}
\end{proof}

\subsection{The Stokes operator on 
	$ \nicefrac{\rL^1(\Omega,\C^d)}{\nabla \rW^{1,1}(\Omega)}$}
\label{ssec:sun dual semigroup}

\

Using \autoref{lem:X'}, we calculate the adjoint of the Stokes operator $A \colon D_0(A) \subset \rC_{\sigma,0}(\Omega)\to\rC_{\sigma,0}(\Omega)$. 

\begin{lem}\label{lem:adjoint}
	Let $\Omega \subset \R^d$ be a bounded $\rC^{1,\alpha}$-domain, $\alpha>0$. Consider the Stokes operator $A \colon D(A) \subset \rC_{\sigma,0}(\Omega)\to\rC_{\sigma,0}(\Omega)$ from \autoref{thm:stokes Linfty}, then its adjoint $A' \colon D_{\nicefrac{M(\Omega)^d}{\nabla\mathrm{BV}(\Omega)}}(A') \subset X' \to X'$ on $X' = \nicefrac{M(\Omega,\C^d)}{\nabla \mathrm{BV}(\Omega)}$ is given by
	\begin{align*}
		A' [\mu] &= [-\Delta \mu+\nabla\pi] ,\quad \textrm{ with domain }\\
        D_{\nicefrac{M(\Omega)^d}{\nabla\mathrm{BV}(\Omega)}}(A') &= \left\{ [\mu] \in  \nicefrac{M(\Omega,\C^d)}{\nabla \mathrm{BV}(\Omega)}\ \colon\ \mu\in \rW^{1,1}_{\sigma,0}(\Omega)\  \textrm{s.t.}\  \exists \pi\in \rL^1(\Omega), \textrm{ s.t. } 
		  \Delta {\mu}-\nabla \pi \in  {M(\Omega,\C^d)} \right\} .
	\end{align*}
\end{lem}
\begin{proof}
	The domain follows from \eqref{eq:domain embeddings} and \autoref{prop:stokes Lp}. 
	Now, consider $[\mu] \in D_{\nicefrac{M(\Omega)^d}{\nabla\mathrm{BV}(\Omega)}}(A')$ and $u \in D_0(A)$. Note that this in particular implies $\div u = 0$ weakly, $u|_{\partial \Omega} = 0$ and $\div (A u) = 0$, $(Au)|_{\partial \Omega} = 0$ weakly. Hence, we obtain for $\mu_1,\mu_2 \in [\mu]$, we have $\mu_1 = \mu_2 + \nabla p\in\rW^{1,1}_0(\Omega,\C^d)$ with $p \in \mathrm{BV}(\Omega)$. This implies
	\begin{equation*}
		\langle A u , \mu_1 \rangle 
		= \langle A u , \mu_2 \rangle + \langle A u , \nabla p \rangle
		= \langle A u , \mu_2 \rangle - \langle \div( A u) , p \rangle + \int_{\partial\Omega} (Au)\cdot\nu\, p \,  \mathrm{d} \sigma = \langle A u , \mu_2 \rangle .
	\end{equation*}
	and therefore $\langle A u , [\mu] \rangle$ is well-defined. By the definition of the adjoint, we see
	\begin{equation*}
		\langle u, A' [\mu] \rangle 
		= \langle A u, [\mu] \rangle .
	\end{equation*}
	By Sobolev embedding, we obtain $[\mu] \in D_{\nicefrac{M(\Omega,\C^d)}{\nabla\mathrm{BV}(\Omega)}}(A') \hookrightarrow \rW^{1,1}_{\sigma,0}(\Omega)$, in particular $\div \mu = 0$ and $\mu |_{\partial \Omega} = 0$. Now, we obtain
	\begin{equation*}
		\langle \nabla \pi, \mu \rangle
		= \int_{\Omega} \nabla \pi \mu \,  \mathrm{d} x 
		= - \int_{\Omega} \pi \div \mu \,  \mathrm{d} x 
		= 0 .
	\end{equation*}
	Finally, Green's formula implies
	\begin{equation*}
		\langle A u, \mu \rangle 
		= \langle -\Delta u + \nabla \pi_u, \mu \rangle 
		= \langle- \Delta u, \mu \rangle
		= \langle u, -\Delta \mu +\nabla \pi\rangle .  \qedhere 
	\end{equation*}
\end{proof}

\begin{rem}
	Note that, by \autoref{prop:stokes Lp}, we have the canonical Sobolev embeddings
	\begin{equation*}
		D_{\nicefrac{M(\Omega,\C^d)}{\nabla\mathrm{BV}(\Omega)}}(A') \hookrightarrow  
		\bigcap_{\frac{d}{p}<\varepsilon<1} \rW^{2-\varepsilon,p'}_{\sigma,0}(\Omega)
		\hookrightarrow
		\bigcap_{\frac{d}{p}<\varepsilon<1} \rW^{2-\varepsilon,1}_{\sigma,0}(\Omega)
		\cap \bigcap_{r < \frac{d}{d-1}} \rW^{1,r}_{\sigma,0}(\Omega)
		\hookrightarrow \rW^{1,1}_{\sigma,0}(\Omega) .
	\end{equation*}
\end{rem}

Now, we are sufficiently prepared to prove our main result. 

\begin{thm}\label{thm:sun-stokes}
	Let $\Omega \subset \R^d$ be a bounded $\rC^{1,\alpha}$-domain, $\alpha>0$.
	The Stokes operator $\tilde{A} \colon D_{1}(\tilde{A}) \subset \nicefrac{\rL^1(\Omega,\C^d)}{\nabla \rW^{1,1}(\Omega)} \to \nicefrac{\rL^1(\Omega,\C^d)}{\nabla \rW^{1,1}(\Omega)}$ defined by
	\begin{align*}
		\tilde{A} [u] &= [-\Delta u+\nabla \pi], 
		\quad \text{ with domain }\\
        \quad D_{1}(\tilde{A})
		&= \left\{ [u],\, u \in  
		\bigcap_{\frac{d}{p}<\varepsilon<1} \rW^{2-\varepsilon,p'}_{\sigma,0}(\Omega)  \colon \exists \pi\in\rL^1(\Omega)\,\textrm{s.t.}\, [\Delta u-\nabla \pi] \in  \nicefrac{\rL^1(\Omega,\C^d)}{\nabla \rW^{1,1}(\Omega)} \right\} ,
	\end{align*}
	is the (negative) generator of a bounded, compact and analytic $\rC_0$-semigroup $(\tilde{T}(t))_{t\geqslant 0}$ of angle $\frac{\pi}{2}$ on $\nicefrac{\rL^1(\Omega,\C^d)}{\nabla \rW^{1,1}(\Omega)}$. Moreover, we have the continuous embedding $\iota \colon \rL^1_{\sigma,n}(\Omega) \hookrightarrow \nicefrac{\rL^1(\Omega,\C^d)}{\nabla \rW^{1,1}(\Omega)}$ with dense range and the semigroup $(\tilde{T}(t))_{t \geqslant 0}$ leaves $\iota (\rL^1_{\sigma,n}(\Omega))$ invariant.
	
\end{thm}
\begin{proof}
	From \autoref{prop:sun dual}, \autoref{thm:stokes Linfty} and \autoref{prop:sun dual stokes} we conclude that the sun-dual operator $A^{\odot} \colon D(A^{\odot}) \subset  \nicefrac{\rL^1(\Omega,\C^d)}{\nabla \rW^{1,1}(\Omega)} \to  \nicefrac{\rL^1(\Omega,\C^d)}{\nabla \rW^{1,1}(\Omega)}$ is the negative generator of a bounded, compact, analytic $\rC_0$-semigroup $(T^{\odot}(t))_{t\geqslant 0}$ of angle $\frac{\pi}{2}$.
	Since $A^{\odot}$ is the part of adjoint $A'$ in $\nicefrac{\rL^1(\Omega,\C^d)}{\nabla \rW^{1,1}(\Omega)}$, we have by \autoref{lem:adjoint} that
	\begin{equation*}
		A^{\odot} [u] = A' [u] = [-\Delta u+\nabla \pi] = \tilde{A} [u] 
	\end{equation*}
	with the domain
	\begin{equation*}
		D_{\nicefrac{M(\Omega)^d}{\nabla\mathrm{BV}(\Omega)}}(A^{\odot})
		= \left\{ [u] \in D_{\nicefrac{M(\Omega,\C^d)}{\nabla\mathrm{BV}(\Omega)}}(A') \colon\exists \pi\in\rL^1(\Omega)\,\textrm{s.t.}\,  [\Delta u-\nabla \pi] \in \nicefrac{\rL^1(\Omega,\C^d)}{\nabla \rW^{1,1}(\Omega)} \right\} = D_{1}(\tilde{A}) .
	\end{equation*}
	Hence, $\tilde{A}$ is the negative generator of a bounded, compact, analytic $\rC_0$-semigroup $(\tilde{T}(t))_{t\geqslant 0}$ of angle $\frac{\pi}{2}$.
	The embedding has been proven in \autoref{lem:X'}. Finally, \autoref{lem:almost L1} implies that $(\tilde{T}(t))_{t \geqslant 0}$ leaves $\iota (\rL^1_{\sigma,n}(\Omega))$ invariant.
\end{proof}

It is worth pointing out that the Stokes semigroup on $\nicefrac{\rL^1(\Omega,\C^d)}{\nabla \rW^{1,1}(\Omega)}$ is an extension of the Stokes semigroups on $\rL^p_{\sigma,n}(\Omega)$. More precisely, the following holds.

\begin{rem}
	For all $p \in (1,\infty)$, we have the canonical identifications embedding $ \rL^p_{\sigma,n}(\Omega) \cong \nicefrac{\rL^p(\Omega,\C^d)}{\nabla \rW^{1,p}(\Omega)}$ and, since $\Omega$ is bounded, a canonical embedding $\nicefrac{\rL^p(\Omega,\C^d)}{\nabla \rW^{1,p}(\Omega)}\hookrightarrow \nicefrac{\rL^1(\Omega,\C^d)}{\nabla \rW^{1,1}(\Omega)} $. We denote their composition by $\iota_p\,:\, \rL^p_{\sigma,n}(\Omega)\hookrightarrow\nicefrac{\rL^1(\Omega,\C^d)}{\nabla \rW^{1,1}(\Omega)}$ and the map is injective and has an unbounded closed and densely defined inverse. In this case, the semigroup $(\tilde{T}(t))_{t \geqslant 0}$ leaves $\iota_p (\rL^p_{\sigma,n}(\Omega))$ invariant. Since the restricted semigroup $(T(t))_{t \geqslant 0}$ with $\iota_p T(t) = \widetilde{T}(t) \iota_p$ is the classical Stokes semigroup on $\rL^p_{\sigma,n}(\Omega)$ and hence a $\rC_0$-semigroup, we obtain by \cite[Section II.2.3]{EN:00} that its generator $A$ is the part of $\tilde{A}$ in $\iota_p (\rL^p_{\sigma,n}(\Omega))$. Hence, as discussed in the introduction, using the isomorphism induced by the Helmholtz decomposition, we see that $A$ is the Stokes operator on $\rL^p_{\sigma,n}(\Omega)$.   
\end{rem}

\subsection{The Stokes operator on $\rL^1_{\sigma,n}(\Omega)$}
\label{ssec:L1-stokes}

\

Finally, we prove the non-generation result for the Stokes operator on $\rL^1_{\sigma,n}(\Omega)$, which shows that \autoref{thm:sun-stokes} is the best result one can expect.

\begin{thm}\label{cor:L_1-stokes}
	Let $\Omega \subset \R^d$ be a bounded $\rC^{1,\alpha}$-domain, $\alpha>0$.
	Consider the Stokes operator $$A \colon D_1(A) \subset \rL^1_{\sigma,n}(\Omega) \to \rL^1_{\sigma,n}(\Omega)$$  as defined in \autoref{cor:Stokes Operator on L1}.
	Then $-A$ does not generate a $\rC_0$-semigroup on  $\rL^1_{\sigma,n}(\Omega)$. 
\end{thm}

\begin{proof} We argue by contradiction. Since $-A$ generates a $\rC_0$-semigroup on $\rL^1_{\sigma,n}(\Omega)$, there exist constants $\lambda_0>0$ and $C>0$ such that
\begin{align*}
\left\|\lambda R(\lambda,-A)\right\|_{\mathcal{L}(\rL^1_{\sigma,n}(\Omega))}\leqslant C,\qquad \lambda\geqslant\lambda_0.
\end{align*}
Equivalently,
\begin{align}
\lambda\left\|R(\lambda,-A)f\right\|_{\rL^1(\Omega)}\leqslant C\left\|f\right\|_{\rL^1(\Omega)}\label{eq: Uniform Resolv Est L1 Proof by contradic Domain}
\end{align}
for every $f\in \rL^1_{\sigma,n}(\Omega)$ and every $\lambda\geqslant\lambda_0$. Our goal is to prove, for $\mu>0$, the estimate
\begin{align}\label{eq: Uniform Resolv Est L1 Proof by contradic Rn+}
    \mu\left\|R(\mu,-A_+)h\right\|_{\rL^1(\R^d_+)}\leqslant C\left\|h\right\|_{\rL^1(\R^d_+)}.
\end{align}
for all $h\in \rC^\infty_{\sigma,c}(\R^d_+)$, where  $A_+$ is the Stokes operator on $\R^d_+$, and derive a contradiction.

\textbf{Step 1:} We prove the half-space estimate for smooth data. Fix $\mu>0$ and such a vector field $h$. Choose $x_0\in\partial\Omega$. After applying an orthogonal transformation $\rQ$, we may assume that the tangent half-space to $\Omega$ at $x_0$ is identified with
\begin{align*}
\R^d_+:=\left\{y=(y',y_d)\in\R^d:y_d>0\right\}.
\end{align*}
For $r>0$, define $\Omega_r:=\left\{y\in\R^d:x_0+r\rQ y\in\Omega\right\}$. The $\rC^{1,\alpha}$ regularity of the boundary implies that $\Omega_r$ converges locally to $\R^d_+$ as $r\downarrow0$. More precisely, for every compact set $K\Subset\R^d_+$, there exists $r_K>0$ such that
\begin{align*}
K\subset\Omega_r,\qquad 0<r<r_K.
\end{align*}
Since $\operatorname{supp}h\Subset\R^d_+$, it follows that  $\operatorname{supp}h\subset\Omega_r$ for all sufficiently small $r>0$.

Define
\begin{align*}
f_r(x):=r^{-d-2}\rQ h\left(\frac{\rQ^\mathsf{T}(x-x_0)}{r}\right),\qquad x\in\Omega.
\end{align*}
For sufficiently small $r>0$, one has $
f_r\in \rC_{\sigma,c}^\infty(\Omega)\subset\rL^2_{\sigma,n}(\Omega)\subset\rL^1_{\sigma,n}(\Omega)$. A change of variables gives
\begin{align*}
\left\|f_r\right\|_{\rL^1(\Omega)}=r^{-2}\left\|h\right\|_{\rL^1(\R^d_+)}.
\end{align*}

Set $\lambda_r:=\frac{\mu}{r^2}$, for sufficiently small $r>0$, one has $\lambda_r\geqslant\lambda_0$. Define $u_r:=R(\lambda_r,-A)f_r$, so that the resolvent estimate on $\Omega$ \eqref{eq: Uniform Resolv Est L1 Proof by contradic Domain} gives
\begin{align*}
\lambda_r\left\|u_r\right\|_{\rL^1(\Omega)}\leqslant C\left\|f_r\right\|_{\rL^1(\Omega)}.
\end{align*}

We now reverse the dilation. Define $v_r(y):=r^d\rQ^\mathsf{T}u_r(x_0+r\rQ y),\qquad y\in\Omega_r$. If $\pi_r$ denotes the pressure associated with $u_r$, define $p_r(y):=r^{d+1}\pi_r(x_0+r\rQ y)$. Then $(v_r,p_r)$ satisfies
\begin{align*}
\left\{\begin{array}{rrl}
     \mu v_r-\Delta v_r+\nabla p_r&=h &\text{in }\Omega_r,\\
\operatorname{div}v_r&=0 &\text{in }\Omega_r,\\
v_r&=0 &\text{on }\partial\Omega_r.
\end{array}\right.
\end{align*}
The normalization in the definition of $v_r$ gives
\begin{align*}
\left\|v_r\right\|_{\rL^1(\Omega_r)}=\left\|u_r\right\|_{\rL^1(\Omega)}.
\end{align*}
Since $\mu=r^2\lambda_r$, it follows that
\begin{align*}
\mu\left\|v_r\right\|_{\rL^1(\Omega_r)}
=r^2\lambda_r\left\|u_r\right\|_{\rL^1(\Omega)}\leqslant Cr^2\left\|f_r\right\|_{\rL^1(\Omega)}=C\left\|h\right\|_{\rL^1(\R^d_+)}.
\end{align*}
Thus,
\begin{align}
\mu\left\|v_r\right\|_{\rL^1(\Omega_r)}\leqslant C\left\|h\right\|_{\rL^1(\R^d_+)},\label{By Contradiction uniform Rescaled L1  Resolvent Estimate}
\end{align}
where the constant $C$ is independent of $r$, $\mu$, and $h$. For smooth solenoidal data, since  $\Omega$ is bounded, $\rL^2\hookrightarrow\rL^1$, so that the $\rL^1$- and $\rL^2$-resolvents on $\Omega$ are consistent. Hence $v_r$ is also the unique energy solution of the above resolvent problem. Testing the equation against $v_r$ gives
\begin{align*}
\mu\left\|v_r\right\|_{\rL^2(\Omega_r)}^2+\left\|\nabla v_r\right\|_{\rL^2(\Omega_r)}^2=\operatorname{Re}\int_{\Omega_r}h\cdot\overline{v_r}.
\end{align*}
By the Cauchy--Schwarz inequality,
\begin{align*}
\operatorname{Re}\int_{\Omega_r}h\cdot\overline{v_r}\leqslant \left\|h\right\|_{\rL^2(\R^d_+)}\left\|v_r\right\|_{\rL^2(\Omega_r)}.
\end{align*}
Consequently,
\begin{align*}
\mu\left\|v_r\right\|_{\rL^2(\Omega_r)}\leqslant \left\|h\right\|_{\rL^2(\R^d_+)}\quad
\text{ and }\quad
\sqrt{\mu}\left\|\nabla v_r\right\|_{\rL^2(\Omega_r)}\leqslant \left\|h\right\|_{\rL^2(\R^d_+)}.
\end{align*}

Extend $v_r$ by zero to all of $\R^d$, and denote this extension by $\widetilde v_r$. Since $v_r\in \rH^1_0(\Omega_r;\C^d)$, one has $\widetilde{v_r}\in \rH^1(\R^d;\C^d)$, together with
\begin{align*}
\left\|\widetilde{v_r}\right\|_{\rL^2(\R^d)}=\left\|v_r\right\|_{\rL^2(\Omega_r)} \quad\text{ and }\quad
\left\|\nabla\widetilde v_r\right\|_{\rL^2(\R^d)}=\left\|\nabla v_r\right\|_{\rL^2(\Omega_r)}.
\end{align*}
Thus, for fixed $\mu$ and $h$, the family $(\widetilde v_r)_r$ is bounded in $\rH^1(\R^d;\C^d)$. After passing to a subsequence, there exists $\widetilde v\in \rH^1(\R^d;\C^d)$ such that
\begin{align*}
\widetilde v_r\rightharpoonup\widetilde v\qquad\text{weakly in }\rH^1(\R^d;\C^d).
\end{align*}
For every $R>0$, the Rellich compactness theorem on the fixed ball $\rB_R$ yields
\begin{align*}
\widetilde v_r\longrightarrow\widetilde v\qquad\text{strongly in }\rL^2(\rB_R;\C^d).
\end{align*}
After a diagonal extraction, we may also assume that
\begin{align*}
\widetilde v_r(y)\longrightarrow\widetilde v(y),\qquad\text{ for almost every $y\in\R^d$}.
\end{align*}

Let $K\Subset\{y_d<0\}$. The local convergence $\Omega_r\to\R^d_+$ implies that $K\cap\Omega_r=\varnothing$ for all sufficiently small $r>0$. Hence, $\widetilde v=0$ almost everywhere in $\{y_d<0\}$. It follows from the characterization by zero extension that
\begin{align*}
v:=\widetilde v|_{\R^d_+}\in \rH^1_0(\R^d_+;\C^d).
\end{align*}
Moreover, since the zero extensions $\widetilde v_r$ are distributionally divergence-free, one has
\begin{align*}
\operatorname{div}v=0\qquad\text{in }\mathcal D'(\R^d_+).
\end{align*}

Let $\varphi\in \rC^\infty_{\sigma,c}(\R^d_+)$. For all sufficiently small $r>0$, one has $\operatorname{supp}\varphi\subset\Omega_r$. The weak formulation for $v_r$ gives
\begin{align*}
\mu\int_{\Omega_r}v_r\cdot\overline{\varphi}+\int_{\Omega_r}\nabla v_r:\nabla\overline{\varphi}=\int_{\R^d_+}h\cdot\overline{\varphi}.
\end{align*}
Passing to the limit yields
\begin{align*}
\mu\int_{\R^d_+}v\cdot\overline{\varphi}+\int_{\R^d_+}\nabla v:\nabla\overline{\varphi}=\int_{\R^d_+}h\cdot\overline{\varphi}.
\end{align*}
By density, this identity holds for every $\varphi\in \rH^{1,2}_{0,\sigma}(\R^d_+)$. The uniqueness of the energy solution therefore implies that $v=R(\mu,-A_+)h$. Now, using Fatou's lemma and the uniform $\rL^1$ estimate for $v_r$ \eqref{By Contradiction uniform Rescaled L1  Resolvent Estimate}, we obtain
\begin{align*}
\mu\left\|R(\mu,-A_+)h\right\|_{\rL^1(\R^d_+)}=\mu\left\|v\right\|_{\rL^1(\R^d_+)}\leqslant\liminf_{r\downarrow0}\mu\left\|v_r\right\|_{\rL^1(\Omega_r)}\leqslant C\left\|h\right\|_{\rL^1(\R^d_+)}.
\end{align*}
We have therefore proved
\begin{align*}
\mu\left\|R(\mu,-A_+)h\right\|_{\rL^1(\R^d_+)}\leqslant C\left\|h\right\|_{\rL^1(\R^d_+)}
\end{align*}
for every $h\in \rC^\infty_{\sigma,c}(\R^d_+)$. 

\textbf{Step 2:} We provide a contradiction. It remains to extend this estimate to arbitrary $h\in\rL^1_{\sigma,n}(\R^d_+)$. We recall that  $\rL^{1}_{\sigma,n}(\R^d_+)=\overline{\rC^\infty_{\sigma,c}(\R^d_+)}^{\lVert\cdot\rVert_{\rL^{1}(\R^d_+)}}$, see \cite[Chap.~4]{BG:25}. Choose $(h_k)_{k\in\N}\subset \rC^\infty_{\sigma,c}(\R^d_+)$ such that
\begin{align*}
\left\|h_k-h\right\|_{\rL^1(\R^d_+)}\xrightarrow[k\to+\infty]{}0.
\end{align*}
Applying the preceding estimate to $h_k-h_\ell$ gives
\begin{align*}
\mu\left\|R(\mu,-A_+)h_k-R(\mu,-A_+)h_\ell\right\|_{\rL^1(\R^d_+)}\leqslant C\left\|h_k-h_\ell\right\|_{\rL^1(\R^d_+)}.
\end{align*}
Thus $\big(R(\mu,-A_+)h_k\big)_{k\in\mathbb N}$ is a Cauchy sequence in $\rL^1(\R^d_+;\C^d)$. To identify its limit with the canonical half-space resolvent, choose $q>1$ sufficiently close to $1$ so that
\begin{align*}
\frac{d}{2}\left(1-\frac{1}{q}\right)<1.
\end{align*}
By Sobolev embeddings, the canonical half-space resolvent estimates, see for instance \cite[Theorem~5.2]{BG:25} combined with \cite{GMS:99}, imply
\begin{align*}
\left\|R(\mu,-A_+)g\right\|_{\rL^q(\R^d_+)}\leqslant C_{\mu,q}\left\|g\right\|_{\rL^1(\R^d_+)},\quad g\in\rL^1_{\sigma,n}(\R^d_+).
\end{align*}
Consequently,
\begin{align*}
R(\mu,-A_+)h_k\xrightarrow[k\to+\infty]{}R(\mu,-A_+)h\qquad\text{in }\rL^q(\R^d_+;\C^d).
\end{align*}
On the other hand, the same sequence converges in $\rL^1(\R^d_+;\C^d)$ to some vector field $u$. Both convergences imply convergence in $\mathcal D'(\R^d_+;\C^d)$, and therefore 
$u=R(\mu,-A_+)h$. Passing to the limit in the estimate for $h_k$ gives
\begin{align*}
\mu\left\|R(\mu,-A_+)h\right\|_{\rL^1(\R^d_+)}\leqslant C\left\|h\right\|_{\rL^1(\R^d_+)}.
\end{align*}
Since $\mu>0$ and $h\in\rL^1_{\sigma,n}(\R^d_+)$ were arbitrary, the proof is complete since it contradicts \cite[Theorem~5.1]{DHP:01}.
\end{proof}

\subsection*{Acknowledgments}

TB acknowledges the support of the DFG Walter-Benjamin Fellowship no. 538212014.
Part of this work was carried out during TB's visit to Princeton University. TB would like to thank Princeton University for its warm hospitality.
The authors completed part of this work during a joint visit to Oberwolfach and would like to thank the MFO for its warm hospitality.


\end{document}